\documentclass[12pt]{amsart}
\usepackage{graphicx} % Required for inserting images

\usepackage{amsmath,amssymb,amsthm, mathrsfs}
\usepackage{enumerate}
\usepackage{hyperref,xcolor}
\usepackage[margin=1.2in]{geometry}
\usepackage{tikz-cd}
\usepackage{comment}
\usepackage{esint}

\numberwithin{equation}{section}
\newtheorem{lemma}{Lemma}[section]
\newtheorem{theorem}[lemma]{Theorem}
\newtheorem{proposition}[lemma]{Proposition}
\newtheorem{corollary}[lemma]{Corollary}

\theoremstyle{definition}
\newtheorem{defn}[lemma]{Definition}
\newtheorem{rmk}[lemma]{Remark}
\newtheorem{notat}[lemma]{Notation}

\newcommand{\Z}{\mathbb{Z}}
\newcommand{\C}{\mathbb{C}}
\newcommand{\QQ}{\mathbb{Q}}
\newcommand{\R}{\mathbb{R}}

\newcommand{\Gal}{\textup{Gal}}
\newcommand{\mfp}{\mathfrak{p}}
\newcommand{\mfa}{\mathfrak{a}}
\newcommand{\mfb}{\mathfrak{b}}
\newcommand{\mfr}{\mathfrak{r}}

\newcommand{\Ind}{\textup{Ind}}
\newcommand{\res}{\textup{res}}

\title{Faithful Artin induction and the Chebotarev density theorem}

\author{Robert J. Lemke Oliver}
\email{robert.lemke\_{}oliver@tufts.edu}
\address{Department of Mathematics, Tufts University, 177 College Ave, Medford, MA 02155}

\author{Alexander Smith}
\email{asmith13@math.ucla.edu}
\address{Department of Mathematics, University of California-Los Angeles, 520 Portola Plaza, Los Angeles, CA 90095}

\begin{document}

\maketitle

\begin{abstract}
Given a finite group $G$, we prove that the vector space spanned by the faithful irreducible characters of $G$ is generated by the monomial characters in the vector space. As a consequence, we show that in any family of $G$-extensions of a fixed number field $F$, almost all are subject to a strong effective version of the Chebotarev density theorem.  We use this version of the Chebotarev density theorem to deduce several consequences for class groups in families of number fields.
\end{abstract}

\section{Introduction}
\subsection{Induction theorems and faithful characters}
Given a finite Galois extension of number fields $K/F$ and a character $\chi\colon \Gal(K/F) \to \C$, there is an  $L$-function associated with $\chi$ referred to as the Artin $L$-function $L(s, \chi)$. In 1924, Artin introduced these functions and showed that $L(s, \chi)^m$ had meromorphic continuation to all of $\C$ for some positive integer $m$. In his paper, Artin developed what is now termed the ``Artin formalism,'' after which the key input to Artin's theorem is the following result in character theory.
\begin{theorem}{\cite[Section 6]{Artin24}}
Given a finite group $G$, any character $\chi\colon G \to \C$ is a $\QQ$-linear combination of characters of the form $\Ind_H^G \psi$, where $H$  ranges over the cyclic subgroups of $G$ and $\psi$ ranges over the linear characters of $H$.
\end{theorem}

In 1947, Brauer proved that $L(s, \chi)$ itself was meromorphic. This too was a direct consequence of a result in character theory:
\begin{theorem}{\cite{Brauer47}}
Given a finite group $G$, any character $\chi\colon G \to \C$ is a $\Z$-linear combination of characters of the form $\Ind_H^G \psi$, where $H$ ranges over the elementary subgroups of $G$ and $\psi$ ranges over the linear characters of $H$.
\end{theorem}
Both of these results, and especially Brauer's theorem, have since assumed a fundamental role in the character theory of finite groups \cite[Chapter 8]{Isaacs76}.

The first major goal of this paper is to prove a strong form of Artin's theorem for faithful characters. Like Artin and Brauer, we will subsequently apply this result to the study of Artin $L$-functions.

\begin{theorem}
\label{thm:faithful}
Given a finite group $G$, any faithful irreducible character $\chi\colon G \to \C$ is a $\QQ$-linear combination of characters of the form $\Ind_H^G \psi$, where $H$  ranges over the nilpotent subgroups of $G$ and $\psi$  ranges over the linear characters of $H$ such that $\Ind_H^G \psi$ is a sum of irreducible faithful characters of $G$.
\end{theorem}
We note that an equivalent condition to $\Ind_H^G \psi$ being a sum of irreducible faithful characters is that 
\begin{equation}
\label{eq:faithful_check}
N \cap H \not \le \ker \psi 
\quad\text{for every nontrivial normal subgroup } N \text{ of  }G.  
\end{equation}

It is convenient to isolate the role of a single normal subgroup $N$ in this condition. This leads to the following definition, which was first considered in work of the first author with Thorner and Zaman \cite{LOTZ}.

\begin{defn}
Given a finite group $G$ and a normal subgroup $N$ of $G$, we write that \emph{Hypothesis $T(G, N)$} holds if every irreducible character $\chi$ whose kernel does not contain $N$ is a a $\QQ$-linear combination of characters of the form $\Ind_H^G \psi$, where $H$  ranges over the  subgroups of $G$ and $\psi$  ranges over the linear characters of $H$ such that 
\[N \cap H \not \le \ker \psi. \]
\end{defn}
The following result then verifies \cite[Conjecture 3.3]{LOTZ}.
\begin{corollary}
\label{cor:TGN}
Given any finite group $G$ and any normal subgroup $N$ of $G$,
Hypothesis $T(G, N)$ holds.
\end{corollary}
It is straightforward to prove this result from Theorem \ref{thm:faithful}. In this paper, though, we proceed in the opposite direction, first proving a variant of this corollary as Theorem \ref{thm:T(G,N)} before showing it implies Theorem \ref{thm:faithful}.

\begin{rmk}
    We have so far been unable to find a counterexample to the natural analogue of Theorem~\ref{thm:faithful} where we restrict our attention to $\Z$-linear combinations of induced characters.  We cautiously expect such a generalization to hold, but it does not directly follow from the ideas behind the proof of Theorem~\ref{thm:faithful}.  
\end{rmk}

\subsection{A Chebotarev density theorem for most field extensions}
In proving Theorem \ref{thm:faithful}, our eventual goal is to prove an averaged form of the Chebotarev density theorem in the style of \cite{PTBW}. Such results take the following form: given a number field $F$ and a group $G$, if $K/F$ is a $G$-extension outside a certain sparse set, we have an effective form of the Chebotarev density theorem for $K/F$. The primary application of this kind of result is to find unconditional proofs of theorems that previously relied on the extended Riemann hypothesis.

To state our averaged form the Chebotarev density theorem, we need to define our sparse set of bad fields. We begin by fixing some basic notation for the paper.
\begin{notat}
Given any number field $F$, we will write $\Delta_F$ for the magnitude of the absolute discriminant of $F$. Given an ideal $\mfa$ of $F$, we write $N\mfa$ for the rational norm of $\mfa$. We write $\pi_F(H)$ for the number of primes of $F$ of rational norm at most $H$.

Given a finite Galois extension $K/F$ and a conjugacy class $C$ of $\Gal(K/F)$, we take $\pi_C(H; K/F)$ to denote the number of primes $\mfp$ of $F$ of norm at most $H$ that are unramified in $K/F$ and whose corresponding Frobenius element $\text{Frob}\, \mfp$ lies in $C$.

Finally, given a character $\chi\colon \Gal(K/F) \to \C$ and a prime $\mfp$ of $F$, we take 
\[\chi(\mfp) = \begin{cases} \chi(\text{Frob}\, \mfp) &\text{ if } K/F \text{ is unramified at }\mfp \\ 0 &\text{ otherwise.}\end{cases}\]
\end{notat}

\begin{defn}
\label{defn:eps_bad}
Given a number field $F$ and a positive number $\epsilon$, we call a finite nontrivial Galois extension $K/F$ an \emph{$\epsilon$-bad extension of $F$} if there is an irreducible faithful character $\chi$ of $\Gal(K/F)$ such that we have
\begin{equation}
\label{eq:bad_defn}
\left|\sum_{N\mfp \le H} \chi(\mfp)\right| \ge \frac{H}{\log H} \cdot \exp\left(- c(\epsilon)\cdot \sqrt{\log H}\right)\quad\text{for some} \quad H \ge (\log \Delta_K)^{2 + \frac{[K:F]}{2\epsilon}},
\end{equation}
where we have taken
\begin{equation}
\label{eq:cepsilon}
c(\epsilon) = \min\left(\frac{\sqrt{\epsilon}}{18},\,\,\frac{1}{29 \cdot [K:\QQ]^{1/2}}\right).
\end{equation}
We take $\mathbb{X}_{\textup{bad}}(F, \epsilon)$ to be the set of $\epsilon$-bad extensions of $F$.
\end{defn}

The main number theoretic result of this paper is the following sparsity result for $\epsilon$-bad extensions of $F$.  It will be proved in Section~\ref{ssec:proof_sparse_bad}.
\begin{theorem}
\label{thm:sparse_bad}
For each number field $F$ and integer $d \geq 2$, there is a constant $C(F,d)>0$ such that for any positive $\epsilon < 1$ and $\Delta \geq 3$, we have
\[\left|\big\{ K \in \mathbb{X}_{\textup{bad}}(F, \epsilon)\,:\,\, \Delta_K \le \Delta\text{ and } [K:F] \le d\big\}\right| \le \Delta^{\epsilon(1 + \delta)} \cdot  (\log \Delta)^{C(F, d)},\]
where we have taken 
\[\delta = C(F, d) \cdot (\log \log \Delta)^{-1/2}.\]
Moreover, we may take $C(F, d) = 400d^2\cdot[F:\mathbb{Q}]$ so long as $\Delta \gg_{F, d} 1$.
\end{theorem}

In short, the number of $\epsilon$-bad extensions of $F$ of degree $d$ is on the order of $\Delta^{\epsilon(1 + o(1))}$ at most.

We will prove the following simple proposition in a strengthened form as Theorem~\ref{thm:avg_cdt}.

\begin{proposition}[The averaged Chebotarev density theorem] \label{prop:avg_cdt}
Choose a number field $F$, a nontrivial finite group $G$, and a positive number $\epsilon \le 1$. Choose a $G$-extension $K$ of $F$, and suppose $K$ contains no field in $\mathbb{X}_{\textup{bad}}(F, \epsilon)$.

Then, for any conjugacy class $C$ of $G$, we have

\[\left|\pi_C(H; K/F) \,-\, \frac{|C|}{|G|} \cdot \pi_F(H)\right| \le  \frac{H}{\log H} \cdot \exp\left(- c(\epsilon) \cdot \sqrt{\log H}\right).\]
for all $H \ge (\log \Delta_K)^{2 + \frac{[K:F]}{2\epsilon}}$, where $c(\epsilon)$ is defined by \eqref{eq:cepsilon}.
\end{proposition}

Taken together, Proposition \ref{prop:avg_cdt} and Theorem \ref{thm:sparse_bad} constitute an averaged form of the Chebotarev density theorem as seen in \cite{PTBW} and \cite{LOTZ}. The central advantage of our result over prior work is the fact that $G$ may be an arbitrary finite group, with the previous results all placing substantial hypotheses on this group.  

Another advantage of these statements compared to previous work is that $\epsilon$ may vary with $\Delta$ in Theorem \ref{thm:sparse_bad} since the constant $C(F, d)$ does not depend on $\epsilon$. In particular, this allows for a trade off between the number of ``exceptional fields'' in the set $\mathbb{X}_\mathrm{bad}(F,\epsilon)$ and the range in which the effective Chebotarev density theorem applies.  For example, by taking $\epsilon$ to be proportional to $\frac{\log \log \Delta}{\log \Delta}$, we obtain an effective Chebotarev density theorem that applies as soon as $H$ is a small power of the discriminant (a range of critical interest in \cite{PTBW}), with a much better bound $(\log \Delta)^{O(1)}$ on the number of exceptional fields than was provided by either \cite{PTBW} or \cite{LOTZ}.

\begin{rmk}
Choose a positive integer $d$. If $\epsilon$ is sufficiently small, the subset of $\epsilon$-bad extensions among the Galois extensions of $F$ of degree at most $d$ is sparse. However, the subset of such extensions that \emph{contain} an $\epsilon$-bad extension of $F$ might not be sparse under some orderings. For example, combining the results of \cite{Maki85} and \cite{BhargavaWood08}, one easily sees that a positive proportion of Galois sextic fields ordered by discriminant contain any fixed cyclic cubic field, e.g. $\mathbb{Q}(\zeta_7+\zeta_7^{-1})$.

A consideration of subfields is inevitable. 
A class function $f\colon \Gal(K/F) \to \C$ is uniquely expressible in the form $\sum_L f_L$, where the sum is over the Galois extensions $L/F$ contained in $K$ and each $f_L$ is a complex combination of the faithful irreducible characters of $\Gal(L/F)$. The averaging procedure needs to handle the character sum for each $f_L$ separately; if we average over $K$ in a family of fields containing a fixed intermediate field $L$, the procedure can only handle the contribution from $f_L$ if $L$ has sufficiently small discriminant to apply an unconditional Chebotarev density theorem. 
\end{rmk}

For some of our arithmetic applications, it is also useful to spell out the following version of the prime ideal theorem that holds for extensions disjoint from $\mathbb{X}_\mathrm{bad}(F,\epsilon)$.

\begin{proposition}\label{prop:avg_pit}
    Let $F$ be a number field, let $\epsilon > 0$, let $L/F$ be a finite extension, let $K/F$ be its normal closure, and let $G = \mathrm{Gal}(K/F)$.  Suppose that $L$ is linearly disjoint from every field in $\mathbb{X}_\mathrm{bad}(F,\epsilon)$ contained in $K$.  Then for all $H \geq (\frac{|G|}{2}\log \Delta_L)^{2+\frac{|G|}{2\epsilon}}$, we have
        \[
            \left| \pi_L(H) - \pi_F(H) \right|
                \leq  \frac{H}{\log H} \cdot ([L:F]-1) \cdot\exp\left( - c(\epsilon) \cdot \sqrt{\log H} \right),
        \]
    where $c(\epsilon)$ is defined by \eqref{eq:cepsilon}.  
\end{proposition}

\subsection{Arithmetic applications}

    One of the virtues of a strong effective Chebotarev density theorem is that it affords many pleasant arithmetic consequences.  We highlight a few that may be easily derived from Propositions~\ref{prop:avg_cdt} and \ref{prop:avg_pit}.  We focus our initial attention on a class of fields for which the linear disjointness hypothesis of Proposition~\ref{prop:avg_pit} may be easily handled, namely, the class of primitive extensions.  Recall that a finite extension $L/F$ is called primitive if it admits no nontrivial proper subextensions.  (For example, any extension of prime degree is primitive.)  For any integer $m \geq 2$, we let $\mathscr{F}_{m,F}^\mathrm{prim}$ denote the set of primitive extensions $L/F$ with degree $m$ inside a fixed algebraic closure $\overline{F}$, and for any $Q \geq 1$, we let $\mathscr{F}_{m,F}^\mathrm{prim}(Q) \subset \mathscr{F}_{m,F}^\mathrm{prim}$ be the subset consisting of those $L$ with $\Delta_L \leq Q$.

    We begin with the following application to bounding the $\ell$-torsion subgroups of the class group of a number field.

    \begin{corollary} \label{cor:ellenberg-venkatesh}
        Let $F$ be a number field and $m,\ell \geq 2$ be integers.  Then there is a constant $A$, depending on $F$, $m$, and $\ell$, such that for any $Q \geq 3$ and all but at most $(\log Q)^A$ extensions $L \in \mathscr{F}_{m,F}^\mathrm{prim}(Q)$, there holds for any $\varepsilon > 0$
            \[
                |\mathrm{Cl}(L)[\ell]| \ll_{F,m,\ell,\epsilon} |\Delta_L|^{\frac{1}{2}-\frac{1}{2\ell(m-1)}+\varepsilon}.
            \]
    \end{corollary}

    Recall that the Minkowski bound implies that $|\mathrm{Cl}(L)|\ll_{[L:\mathbb{Q}],\varepsilon} |\mathrm{Disc}(L)|^{\frac{1}{2}+\varepsilon}$ for every $\varepsilon > 0$.  Thus, Corollary~\ref{cor:ellenberg-venkatesh} obtains an improvement over the trivial bound $|\mathrm{Cl}(L)[\ell]| \leq |\mathrm{Cl}(L)| \ll _{[L:\mathbb{Q}],\varepsilon} |\mathrm{Disc}(L)|^{\frac{1}{2}+\varepsilon}$ for almost all $L \in \mathscr{F}_{m,F}^\mathrm{prim}$.  Analogous improvements were also obtained in \cite{LOTZ,PTBW} for certain subsets of $\mathscr{F}_{m,F}^\mathrm{prim}$ defined by particular Galois and inertial conditions.  Corollary~\ref{cor:ellenberg-venkatesh} improves over these prior results in two ways.  First, and most substantially, it removes these auxiliary Galois and inertial conditions.  Second, it refines the bound on the number of $L$ to which the result does not apply from a bound of the form $O_{F,m,\ell,\varepsilon}(Q^\varepsilon)$ to the bound $(\log Q)^A$ as in its statement.

    Beyond the ``almost all'' result of Corollary~\ref{cor:ellenberg-venkatesh}, one may ask for bounds on the moments of $|\mathrm{Cl}(L)[\ell]|$ as $L$ varies in a family such as $\mathscr{F}_{m,F}^\mathrm{prim}(Q)$.  For this, using machinery of Koymans and Thorner \cite{KoymansThorner} (which builds on and generalizes work of Heath-Brown and Pierce \cite{HBP} and Frei and Widmer \cite{FreiWidmer}), we obtain the following.

    \begin{corollary} \label{cor:koymans-thorner}
        Let $F$ be a number field, $m \geq 2$ be an integer, $Q \geq 1$, and $r \geq 1$ be an integer.  Then for every integer $\ell \geq 2$ and every $\varepsilon > 0$, there holds
            \[
                \sum_{ L \in \mathscr{F}_{m,F}^\mathrm{prim}(Q)} |\mathrm{Cl}(L)[\ell]|^r
                    \ll_{F,m,\ell,\varepsilon} Q^{\frac{r}{2}+\varepsilon} \cdot \left( 1 + |\mathscr{F}_{m,F}^\mathrm{prim}(Q)|^{1- \frac{r}{\ell(m-1)+1}}\right).
            \]
    \end{corollary}

    The main theorem of \cite{KoymansThorner} proves Corollary~\ref{cor:koymans-thorner} in the case that $m$ is prime.  For general $m$, they prove a result analogous to Corollary~\ref{cor:koymans-thorner}, but only for the subset $\mathscr{F}_{m,F}^{S_m} \subseteq \mathscr{F}_{m,F}^\mathrm{prim}$ consisting of extensions $L/F$ whose normal closure over $F$ has Galois group $S_m$.  We also obtain an analogue of this result for any primitive permutation group $G$ of degree $m$ and the associated subset $\mathscr{F}_{m,F}^G$; see Corollary~\ref{cor:koymans-thorner-group} below.
    
    As another sample application to class groups, we have the following variant of a well known result of Bach \cite{Bach} on the generation of the ideal class group assuming the generalized Riemann hypothesis.

    \begin{theorem} \label{thm:approximate-bach}
        Let $F$ be a number field, $m \geq 2$ an integer, $Q \geq 3$, $\varepsilon>0$, and $\ell >m$ be prime.  Then with at most $O_{F,m,\ell,\varepsilon}(Q^{\varepsilon})$ exceptions, each $L \in \mathscr{F}_{m,F}^\mathrm{prim}(Q)$ is such that $\mathrm{Cl}(L) / \ell \mathrm{Cl}(L)$ is generated by primes in $L$ of norm at most $(\log Q)^{3 \ell^{2m} (m!)^2/\varepsilon}$.
    \end{theorem}

    Under GRH, Bach's work \cite[Theorem 4]{Bach} implies that the class group $\mathrm{Cl}(L)$ of any number field $L$ is generated by primes of norm at most $12 (\log \Delta_L)^2$.  Thus, the conclusion of Theorem~\ref{thm:approximate-bach} is weaker than this both in terms of the size of the primes required and in that it only concerns the co-$\ell$ part of the class group (though of course its conclusion is unconditional).  However, we note that the $\ell$-torsion conjecture on class groups (that for a fixed $\ell \geq 2$, $|\mathrm{Cl}(L)[\ell]| \ll_{[L:\mathbb{Q}],\ell,\varepsilon} \Delta_L^{\varepsilon}$ for every $\varepsilon>0$ and every number field $L$) is equivalent to asserting that the rank of $\mathrm{Cl}(L)[\ell]$ (which equals the rank of $\mathrm{Cl}(L)/\ell \mathrm{Cl}(L)$) is $o_{[L:\mathbb{Q}],\ell}( \log \Delta_L)$.  Theorem~\ref{thm:approximate-bach} does not provide an improvement even over the trivial bound on the rank of $\mathrm{Cl}(L)/\ell \mathrm{Cl}(L)$, but we nevertheless consider it an interesting result in its own right that is reflective of what is currently possible toward this line of attack on the $\ell$-torsion conjecture.

    Finally, returning to Artin $L$-functions themselves, we also provide essentially GRH quality bounds on the values $L(1,\chi)$ for almost all Artin $L$-functions $L(s,\chi)$.

    \begin{corollary} \label{cor:artin-bound}
        Let $F$ be a number field and let $G$ be a finite group.  Let $Q \geq 1$ and $\varepsilon>0$.  Then, apart from at most $O_{F,G,\varepsilon}(Q^\varepsilon)$ exceptional fields $K$, each faithful, irreducible Artin $L$-function $L(s,\chi)$ attached to $\mathrm{Gal}(K/F)$ for a Galois $G$-extension $K$ with $\Delta_K \leq Q$ satisfies
            \[
                (\log \log Q)^{\min\{ \Re(\chi(g)) : g \in G\}}
                    \ll_{F,G,\varepsilon} L(1,\chi)
                    \ll_{F,G,\varepsilon} (\log\log Q)^{\chi(1)}.
            \]
    \end{corollary}

    Bounds of the same quality (though with sharper implied constants) follow from GRH, and in many cases are known to be close to optimal.  See, for example, \cite[Theorem 2]{Duke03}.

\subsection{An overview of our methods}
\subsubsection{Faithful Artin induction}
 As in the partial result \cite[Theorem 5.6]{LOTZ}, our proof of hypothesis $T(G, N)$ and the related stronger hypothesis $T_0(G, N)$ introduced below in Definition \ref{defn:T0GN} is by induction on the order of $G$. This approach lets us assume that that hypotheses $T_0(H, N \cap H)$ hold for all proper subgroups $H$ of $G$. The proof of hypothesis $T_0(G, N)$ then reduces to proving that the elements in a given coset of $N$ are connected by chains of proper subgroups of $G$; see Proposition \ref{prop:TGN_induction} for details.

In our proof of Proposition \ref{prop:TGN_induction} in Section \ref{ssec:TGN_induction}, we explicitly construct such chains of proper subgroups. Our construction is $p$-local for some prime $p$ dividing $[G: N]$, with the involved groups being normalizers of $p$-subgroups.

The proof of this hypothesis establishes Theorem \ref{thm:faithful} in the case that $G$ has a unique minimal normal subgroup. More generally, we establish this theorem by considering the socle of $G$, which is the subgroup of $G$ generated by its minimal normal subgroups. The socle is known to be a direct product of characteristically simple groups, and we take advantage of this decomposition to prove the general case of Theorem \ref{thm:faithful}.

\subsubsection{The Chebotarev density theorem in families}
Once Theorem \ref{thm:faithful} is proved, we may turn to number theory. Our first result here is Theorem \ref{thm:bilinear_basic}, which is a bilinear character sum estimate for the coefficients of Artin $L$-functions corresponding to direct sums of monomial characters. The proof of this result uses standard techniques; we first prove a smoothed character sum estimate for such coefficients using the convexity bounds for these $L$-functions, then derive a large sieve from this in the typical way.

The novelty of our approach instead comes in Theorem \ref{thm:Holder}, where we use this large sieve to prove bilinear estimates for shorter character sums using H\"{o}lder's inequality. This technique originates in work of Friedlander and Iwaniec \cite[(21.9)]{Fried98}, but it has not previously been used in this context. The H\"{o}lder trick sidesteps the consideration of zero free regions in families of \cite{PTBW} and \cite{LOTZ} and leads to results of a similar quality to \cite{LOTZ} in greater generality. The disadvantage of this approach is that we are left with no results about zero free regions in families.

With an eye to future applications where such concreteness may matter, we have made an effort to keep the constants appearing in this paper explicit. At the same time, we have made no effort to optimize these constants.

\subsubsection{Layout}
In Section \ref{sec:TGN}, we define the hypothesis $T_0(G, N)$ and show that it holds for all $(G, N)$. We then use this fact to prove Theorem \ref{thm:faithful} in Section \ref{sec:faithful}.

In Section \ref{sec:Lbound}, we prove a smooth character sum estimate for certain $L$-functions. We use this to prove bilinear character sum estimates in Section \ref{sec:Bibound}. In Section \ref{sec:averaged}, we combine these estimates with Theorem \ref{thm:faithful} and the unconditional Chebotarev density theorem to prove Theorem \ref{thm:sparse_bad}. Finally, in Section \ref{sec:arithmetic}, we give some arithmetic applications of our work.

\section*{Acknowledgements}
    The authors would like to thank Jesse Thorner and Asif Zaman for useful conversations.

    RJLO was supported by NSF grant DMS 2200760 and by a Simons Foundation Fellowship in Mathematics.  This research was conducted during the period that AS served as a Clay Research Fellow.

\section{The hypothesis $T_0(G, N)$}
\label{sec:TGN}
Rather than working with irreducible characters as in the statement of Theorem \ref{thm:faithful}, it is convenient to focus on class functions, leading to the following definition.
\begin{defn}
\label{defn:RI}
Choose a finite group $G$ and a set $\{N_1, \dots, N_k\}$ of normal subgroups of $G$. We take
\[  \mathcal{R}(G;\{N_1,\dots,N_k\}) \]
to be the $\C$-vector space of complex class functions $f: G \to \C$ of $G$ whose push forward to any $G/N_i$ is $0$. That is, a class function $f: G \to \C$ lies in this space if and only if
\[\sum_{g \in \sigma N_i} f(g) = 0 \text{ for all $\sigma \in G$ and each $1 \leq i \leq k$}.\]
We take $\mathcal{I}(G; \{N_1, \dots, N_k\})$ to be the subspace of this vector space spanned by the characters of the form $\Ind_H^G \psi$, where $H$ ranges over the nilpotent subgroups of $G$ and $\psi$ ranges over the linear characters of $H$ satisfying
\[H \cap N_i \not \le \ker \psi \quad\text{for } i\le k.\]
To simplify notation, we will alternatively write these spaces in the form $ \mathcal{R}(G;\,N_1,\dots,N_k)$ and $ \mathcal{I}(G;\,N_1,\dots,N_k)$.
\end{defn}

\begin{defn}
\label{defn:T0GN}
Given a finite group $G$ and a normal subgroup $N$ of $G$, we  write that \emph{Hypothesis $T_0(G, N)$ holds} if
\[\mathcal{R}(G; N) = \mathcal{I}(G; N).\]
\end{defn}
This is stronger than the hypothesis $T(G, N)$ because of our restriction to nilpotent subgroups in Definition \ref{defn:RI}. The aim of this section is to prove the following:
\begin{theorem}
\label{thm:T(G,N)}
Hypothesis $T_0(G, N)$ holds for all finite groups $G$ and all normal subgroups $N$ of $G$.
\end{theorem}

\subsection{First reductions for Theorem \ref{thm:T(G,N)}}
\label{ssec:first_reductions}
Given two class functions $f_1, f_2$ on $G$, we  will define an inner product $\langle f_1, f_2 \rangle$ using the standard formula 
\[\langle f_1, f_2 \rangle = \frac{1}{|G|} \sum_{\sigma \in G} f_1(\sigma) \overline{f_2(\sigma)}.\]
As in \cite{LOTZ}, we will reframe hypothesis $T_0(G, N)$ in terms of the orthogonal complement of $\mathcal{I}(G; N)$ with respect to this product.

\begin{lemma}
\label{lem:TGN_dual}
Hypothesis $T_0(G, N)$ holds if and only if every class function $f: G \to \C$ in $\mathcal{I}(G; N)^{\perp}$ is constant on each coset of $N$.
\end{lemma}
\begin{proof}
Since the inner product on class functions is a perfect pairing, hypothesis $T_0(G, N)$ is equivalent to the claim
\[\mathcal{I}(G; N)^{\perp} \subseteq \mathcal{R}(G; N)^{\perp}.\]
But the condition for a class function $f$ to lie in $\mathcal{R}(G; N)$ can be expressed in the form 
\begin{equation}
\label{eq:RGNperp}
\frac{1}{|G|}\sum_{\sigma \in G} f(\sigma) \overline{g(\sigma)} = 0 \quad\text{for all}\quad g: G/N \to \C.
\end{equation}
If we take $\widetilde{g}$ to be the class function on $G$ given by the formula
\[\widetilde{g}(\sigma) = \frac{1}{|G|}\sum_{\tau \in G} g(\tau\sigma \tau^{-1}), \]
we see that the left hand side of \eqref{eq:RGNperp} equals $\langle f, \widetilde{g}\rangle$. So \eqref{eq:RGNperp} gives that $\mathcal{R}(G; N)^{\perp}$ is the set of class functions on $G$ coming from class functions on $G/N$. This gives the lemma.
\end{proof}

Our proof of Theorem \ref{thm:T(G,N)} is by induction, with the induction step handled by the following proposition. This reduction can be seen in the proof of \cite[Theorem 5.6]{LOTZ}.

\begin{proposition}
\label{prop:TGN_induction}
Let $G$ be a finite group with trivial center, and let $N$ be a normal subgroup such that $G/N$ is cyclic and such that
\[T_0(H, H \cap N)\quad\text{holds for all proper subgroups} \quad H < G. \]
Then, given any class function $f$ of in $\mathcal{I}(G; N)^{\perp}$ and any two elements $x, y$ of $G$ in the same coset of $N$, we have $f(x) = f(y)$.
\end{proposition}

\begin{proof}[Proof that Proposition \ref{prop:TGN_induction} implies Theorem \ref{thm:T(G,N)}]
Suppose hypothesis $T_0(G, N)$ did not hold for some $G$  and $N$, and choose $(G, N)$ with $G$ of minimal order and, given $G$, with $N$ of minimal order so this hypothesis is not satisfied. Following \cite[Lemma 5.7(i)]{LOTZ}, we see that we may assume that $G/N$ is cyclic.

Suppose first that $Z(G)$ is nontrivial. By observing that nilpotent subgroups of $G/Z(G)$ have nilpotent preimage in $G$, the argument for \cite[Lemma 5.7 (ii)]{LOTZ} shows that 
\[T_0(G/(N \cap Z(G)), N/(N \cap Z(G))\quad\text{and}\quad T_0(G, N \cap Z(G))\]
together imply $T_0(G, N)$. By the induction hypothesis, we must either have $N \cap Z(G) = 1$ or $N \cap Z(G) = N$. We can rule out the former case using the argument of \cite[Theorem 3.7]{LOTZ}.

In the latter case, $G$ is a cyclic extension of a central subgroup and is hence nilpotent, so $T_0(G, N)$ is equivalent to $T(G, N)$, and we reach a contradiction by \cite[Theorem 5.6]{LOTZ}. So we must have $Z(G) = 1$.

By Lemma \ref{lem:TGN_dual}, there are some $f$ in $\mathcal{I}(G; N)^{\perp}$ and some $x, y \in G$ in the same coset of $N$ such that 
\[f(x) \ne f(y).\]
At the same time, by the minimality of the order of $G$, we know that $T_0(H, H \cap N)$ holds for every proper subgroup $H$ of $G$, so Proposition \ref{prop:TGN_induction} gives
\[f(x) =f(y).\]
This is a contradiction, so $T_0(G, N)$ holds for all $(G, N)$.
\end{proof}

The power of reframing Theorem \ref{thm:T(G,N)} in this way comes from the following lemma.

\begin{lemma}
\label{lem:subgroup_game}
Suppose $(G, N)$ satisfies the hypotheses of Proposition \ref{prop:TGN_induction}. Then, for any proper subgroup $H$ of $G$, any class function $f$ in $\mathcal{I}(G; N)^{\perp}$, and any $x, y \in H$ in the same coset of $H \cap N$, we have
\[f(x) = f(y).\]
\end{lemma}
\begin{proof}
Since $T_0(H, H \cap N)$ holds by assumption, this follows from Lemma \ref{lem:TGN_dual}.
\end{proof}

The following technical lemma reduces the proof of Proposition \ref{prop:TGN_induction} to a case where the pair $x, y$ obeys a weak niceness property.
\begin{lemma}
\label{lem:TGN_edge}
Choose $(G, N)$ satisfying the hypotheses of Proposition \ref{prop:TGN_induction}. Take $F(G)$ to be the Fitting subgroup of $G$.

Suppose that, given any $f$, $x$, $y$ as in Proposition \ref{prop:TGN_induction}, we have $f(x) = f(y)$ whenever there is some prime $p \mid [G:N]$ dividing the orders of $x$ and $y$ in $G/F(G)$. Then the conclusion of Proposition \ref{prop:TGN_induction} holds for $(G, N)$.
\end{lemma}
\begin{proof}
    Choose $f$, $x$, and $y$ as in Proposition \ref{prop:TGN_induction}. Our aim is to show that $f(x) = f(y)$.  We may assume that $x$ and $y$ generate $G$, i.e. $G = \langle x,y\rangle$, since  Lemma \ref{lem:subgroup_game}  would otherwise imply that $f(x)=f(y)$. Since $x$ and $y$ lie in the same coset of $N$, we also find that $G/N$ is generated by $xN$.

    By the argument of \cite[Theorem 5.6]{LOTZ}, we find there is $m \ge 1$ coprime to $[G: N]$ such that $f(x) = f(x^m)$ and $f(y) = f(y^m)$ and such that $x^m$ and $y^m$ have orders divisible only by primes dividing $[G: N]$. So we may assume $x$ and $y$ have orders divisible only by primes dividing $[G: N]$.

    Now suppose that $F(G) \cap xN$ is empty, so $x$ maps to a nontrivial element of $G/ N \cdot F(G)$. Choose a prime $p$ dividing the order of $[G: N \cdot F(G)]$. Since $xN = yN$ generates $G/N$, we see that the images of $x$ and $y$ generate $G/ N \cdot F(G)$, so the images of $x$ and $y$ in this quotient, and hence also in the quotient $G/F(G)$, have order divisible by $p$. By the assumptions of the lemma, we thus have $f(x) = f(y)$.

So we may assume that $F(G)$ meets $xN$. This implies that $F(G)$ surjects onto the quotient $G/N$.

    If $F(G) = G$, then $G$ must be nilpotent since $F(G)$ is.  In this case, the condition that $Z(G) = 1$ implies that $G=1$, and the conclusion follows.  Hence, we may assume that the index $[G:F(G)]$ is not $1$.  We may further reduce to the case that the indices $[G:F(G)]$ and $[G:N]$ are not coprime. Otherwise, both $x$ and $y$ would have trivial image in the quotient $G/F(G)$ by our assumption on their orders. This would imply that $x$ and $y$ both lie in  $F(G)$, so Lemma \ref{lem:subgroup_game} would give $f(x)=f(y)$.  

   So we may assume that there is some prime $p$ dividing both $[G:F(G)]$ and $[G:N]$.  In this case, we claim that $[G:F(G)]$ is not a power of $p$, or, equivalently, that $G/F(G)$ is not a $p$-group.  To prove this, we note that $p \mid |F(G)|$ since $F(G)$ surjects onto $G/N$.  Because $F(G)$ is nilpotent, it follows that its center $Z(F(G))$ has order divisible by $p$ as well.  Take $P$ to be the Sylow $p$-subgroup of $Z(F(G))$.

   Then $G/F(G)$ acts on $P$ via conjugation. Since we assumed $Z(G) =1$, the only fixed point of this action can be $1$.   This implies that $G/F(G)$ cannot be a $p$-group, as a $p$-group acting on a finite nontrivial $p$-group always fixes some element besides $1$. So $[G: F(G)]$ is not a power of $p$.

    Recall that we have assumed $F(G) \cap xN$ is non-empty, so we may fix some $z$ in this intersection.  We now claim that we must have $f(x) = f(z)$. Applying this claim to $y$ will show that $f(y) = f(z) = f(x)$, so this claim suffices to prove the lemma.

    Consider the order of $x$ in $G/F(G)$.  If this order is a power of the prime $p$ chosen above, then $\langle x, F(G) \rangle$ must be a proper subgroup of $G$ since $[G:F(G)]$ is not a power of $p$.  As this subgroup contains both $x$ and $z$, we conclude that $f(x)=f(z)$ by Lemma \ref{lem:subgroup_game}. 

Thus, we may assume that there is some prime $q \ne p$ dividing the order of $x$ in $G/F(G)$.  Choose $m \geq 1$ so that $x^m$ has order $q$ in $G/F(G)$, and take $w = z (x z^{-1})^m$.  Since $z$ lies in $xN$ and $F(G)$, we have
\[wN = zN = xN\quad\text{and}\quad w F(G) = x^m F(G).\]
 From the hypotheses of the lemma, we therefore find that $f(x) = f(w)$.  But since $w$ has order $q$ in $G/F(G)$ and the index $[G:F(G)]$ is divisible by $p \neq q$, the subgroup $\langle w, F(G) \rangle$ is proper and contains both $w$ and $z$, so we find $f(w) = f(z)$ by Lemma \ref{lem:subgroup_game}, and hence $f(x)=f(z)$.  This gives the claim, and the lemma follows.
\end{proof}

\subsection{The proof of Proposition \ref{prop:TGN_induction}}
\label{ssec:TGN_induction}

Our proof of the general case of Proposition \ref{prop:TGN_induction} takes advantage of the $p$-local subgroup structure of $G$. This approach requires the following two lemmas.

    \begin{lemma}\label{lem:p-normalizers-grow}
        Let $G$ be a finite group and let $P \leq G$ be a $p$-group.  Then either $P$ is a Sylow $p$-subgroup of $G$ or $N_G(P)$ contains a $p$-group properly containing $P$.
    \end{lemma}
    \begin{proof}
        Take $S \subseteq G$ to be a Sylow $p$-subgroup of $G$ containing $P$. Since all subgroups of nilpotent groups are subnormal \cite[Lemma 2.1]{Isaacs}, $N_S(P)$ either strictly contains $P$ or $S = P$.
    \end{proof}

\begin{lemma}\label{lem:sylows-linked}
    Choose a finite group $G$ and a normal subgroup $N$ of $G$ such that $G/N$ is cyclic. Then, given a prime $p$  and a Sylow $p$-subgroup $S$ of $G$, we have
\[N_G(S) / (N_G(S) \cap N) \simeq G / N.\]
\end{lemma}
   \begin{proof}
        Observe that $S \cap N$ is a Sylow $p$-subgroup of $ N$.  The Frattini argument \cite[Lemma 1.13]{Isaacs} shows that $G = N \cdot  N_G(S \cap N)$. So we may choose $x$ in $N_G(S \cap N)$ that maps to a generator of $G/N$.

        Taking $\langle x \rangle_p$ to be the maximal $p$-subgroup of $\langle x \rangle$, we thus see that $\langle x \rangle_p \cdot (S \cap N)$ is a Sylow $p$-subgroup of $G$ and is normalized by $x$. This subgroup is conjugate to $S$, so we find that some conjugate of $x$ normalizes $S$.
    \end{proof}

\begin{proof}[Proof of Proposition \ref{prop:TGN_induction}]
By Lemma \ref{lem:TGN_edge}, it suffices to assume that there is a prime $p \mid [G: N]$ such that the images of $x$ and $y$ in $G/F(G)$ have orders divisible by $p$. Take $F(G)_p$ to be the maximal $p$-subgroup of $F(G)$.

With this $p$ fixed, take $\langle x \rangle_p$ to be the maximal $p$-subgroup of $\langle x \rangle$, and take \[P_0(x) = \langle x \rangle_p \cdot F(G)_p.\]
Supposing $P_i(x)$ has been defined for a given $i \ge 0$, we then define
\[G_i(x) = N_G(P_i(x)),\]
and we take $P_{i+1}(x)$ to be a Sylow $p$-subgroup of $G_i(x)$ containing $P_i(x)$. This defines sequences of groups 
\[G_0(x), G_1(x), \dots\quad\text{and}\quad P_0(x), P_1(x), \dots.\]
We note that $x$ is contained in $G_0(x)$.

By Lemma \ref{lem:p-normalizers-grow}, we have 
\[P_i(x)\le P_{i+1}(x) \quad\text{for} \quad i \ge 0,\]
with equality only if $P_i(x)$ is a Sylow $p$-subgroup of $G$. So we may fix $k \ge 0$ such that $P_k(x)$ is a Sylow $p$-subgroup of $G$.

We also have 
\[F(G)_p < P_0(x) \le P_i(x),\]
so no $P_i(x)$ is a normal subgroup of $G$. This means that $G_i(x)$ is a proper subgroup of $G$ for all $i \ge 0$.

We claim that the projection from $G_i(x)$ to $G/N$ is surjective for all $i \ge 0$. It is true for $i = 0$. Now suppose it is true for $G_{i}(x)$, and that we wish to prove it for $G_{i+1}(x)$. By applying Lemma \ref{lem:sylows-linked} to the extension $G_i(x)/N \cap G_i(x)$ with Sylow subgroup $P_{i+1}(x)$, we find that the intersection
\[G_i(x) \cap G_{i+1}(x) \cap xN\]
is nonempty. This gives the claim by induction, and we may take $x_i$ to be some element in this intersection for any $i \ge 0$.

Applying Lemma \ref{lem:subgroup_game} to $G_0(x), G_1(x), \dots$ gives
\[f(x) = f(x_0) = f(x_1) = \dots = f(x_k).\]
Note that the element $x_k$ lies in $N_G(S)$ for some Sylow $p$-subgroup $S$ of $G$. 

Applying the same argument for $y$ shows that there is $y' \in G$ so $y'$ lies in some conjugate of $N_G(S)$ and $f(y) = f(y')$. Take $z$ to be a conjugate of $y'$ in $N_G(S)$. Then
\begin{alignat*}{2}
f(x) = f(x_k) &= f(z) &&\text{by Lemma \ref{lem:subgroup_game} for $N_G(S)$} \\ & = f(y') = f(y) \qquad&&\text{since $f$ is a class function,} 
\end{alignat*}
giving the proposition.
\end{proof}

\begin{rmk}
Given a nilpotent group $G$ and a normal subgroup $N$ of $G$, we claim that $\mathcal{R}(G; N)$ is spanned by the collection of subspaces $\mathcal{R}(H; H \cap N)$, where $H$ varies over the elementary subgroups of $G$ of rank at most $2$. (Recall that the rank of a group is the minimal number of generators.)  
Following the argument after Proposition \ref{prop:TGN_induction}, we find that this claim reduces to showing that, given a cyclic extension $G/N$ with $G$ nilpotent, given an element $f$ in $\mathcal{R}(G;\, N)$ orthogonal to the sum of such $\mathcal{R}(H; H \cap N)$, and given $x, y \in G$ in the same coset of $N$, we have $f(x) = f(y)$. 

To prove this claim, take $w = yx^{-1}$, and choose a sequence of integers $b_0 = 0, b_1, \dots, b_{k-1}, b_k = 1$ such that $w^{b_{i+1} - b_{i}}$ has prime power order for each $i < k$. Then
\[\left\langle w^{b_i} x,\, w^{b_{i+1} - b_i} \right\rangle\]
is an elementary group of rank at $2$ for $i < k$, 
so the orthogonality assumption gives
\[f(w^{b_i}x) = f(w^{b_{i+1} -b_i} w^{b_i} x) = f(w^{b_{i+1}x}) \quad\text{for }\, i < k.\]
So 
\[f(x) = f(w_0x) = f(w_kx) = f(y),\]
giving the claim.

As a consequence, we find that $T_0(G, N)$ still holds for all $(G, N)$ even if we replace the nilpotent groups $H$ in the definition of $\mathcal{I}(G; N)$ with elementary subgroups of rank $\le 2$.
\end{rmk}

\section{The proof of Theorem \ref{thm:faithful}}
\label{sec:faithful}
The proof of Theorem \ref{thm:faithful} reduces to showing
\[\mathcal{R}(G; N_1, \dots, N_k) = \mathcal{I}(G; N_1, \dots, N_k)\]
for any group $G$, where $N_1, \dots, N_k$ enumerates the  minimal normal subgroups of $G$. To do this, we will give a sequence of lemmas that reduce what we need to show to cases we have already dealt with in Section \ref{sec:TGN}.
\begin{lemma}
\label{lem:faith_cyclic}
Take $G/N$ to be an extension of finite groups, and take $N_1, \dots, N_k$ to be normal subgroups of $G$ contained in $N$. Then
\[\mathcal{R}(G;\, N_1, \dots, N_k) = \sum_{H} \textup{Ind}^G_H \left(\mathcal{R}(H;\, N_1, \dots, N_k)\right),\]
where the sum is over all subgroups $H$ of $G$ containing $N$ such that $H/N$ is cyclic.

In particular, if
\[\mathcal{R}(H;\, N_1, \dots, N_k) = \mathcal{I}(H;\, N_1, \dots, N_k)\]
for all such $H$, then
\[\mathcal{R}(G;\, N_1, \dots, N_k) = \mathcal{I}(G;\, N_1, \dots, N_k).\]
\end{lemma}
\begin{proof}
Take $\mathcal{C}$ to be the set of subgroups $H$ of $G$ containing $N$ such that $H/N$ is cyclic. By inclusion-exclusion, for each $H \in \mathcal{C}$, there is an integer $a_H$ such that 
\[\sum_{H \in \mathcal{C}} a_H \delta_{g \in H} = 1 \quad\text{for all}\quad g \in G,\]
where $\delta_{g \in H}$ denotes the Kronecker delta.

As a result, given $f \in \mathcal{R}(H; \, N_1, \dots, N_k)$, we have
\[f = \sum_{H \in \mathcal{C}} a_H f_H,\]
where $f_H$ denotes the restriction of $f$ to $H$. This implies
\[f = \sum_{H \in \mathcal{C}} a_H \frac{|H|}{|G|} \Ind_H^G f_H.\]
Since $f_H$ lies in $\mathcal{R}(H;\, N_1, \dots, N_k)$, the first claim follows. The second then follows from the simple observation
\[\mathcal{I}(G; \,N_1, \dots, N_k) \supseteq \sum_{H \in \mathcal{C}} \Ind_{H}^{G}\left(\mathcal{I}(H;\, N_1, \dots, N_k)\right).\]
\end{proof}

\begin{defn}
Take $G/N$ to be a cyclic extension of finite groups, and choose $\sigma \in G$. Take $\mathscr{N}$ to be a set of normal subgroups of $G$ contained in $N$.  We then define
\[\mathcal{R}(\sigma N;\, \mathscr{N})\]
to be the subspace of $\mathcal{R}(G;\, \mathscr{N})$ of functions that are $0$ outside the coset $\sigma N$.
\end{defn}

Suppose $\sigma$ generates $G/N$. Given $j \ge 0$, take $g_j$ to be the function on $G$ that is $1$ on $\sigma^j N$ and $0$ outside $\sigma ^j N$. Then $f \cdot g_j$ lies in $\mathcal{R}(\sigma^j N;\, \mathscr{N})$ for any $f \in \mathcal{R}(G;\, \mathscr{N})$. Since 
\[1 = \sum_{j = 1}^{[G: N]} g_j,\]
we thus have
\[\mathcal{R}(G;\, \mathscr{N}) = \bigoplus_{j=1}^{[G: N]} \mathcal{R}(\sigma^j N;\, \mathscr{N}). 
 \]
Given a subgroup $H$ of $G$ containing $N$, we have
\[\text{Ind}_H^G\left(\mathcal{R}(H;\, \mathscr{N})\right) =\bigoplus_{\substack{j \le [G: N]\\ \sigma^j \in H }} \mathcal{R}(\sigma^j N;\, \mathscr{N}),\]
so we find 
\begin{equation}
\label{eq:cyclic_decomp}
\mathcal{R}(G;\, \mathscr{N}) = \bigoplus_{j \,\perp\, [G: N]} \mathcal{R}(\sigma^j N;\, \mathscr{N}) \oplus \sum_{H} \text{Ind}_H^G\left(\mathcal{R}(H;\, \mathscr{N})\right),
\end{equation}
 where the direct sum is over positive integers $j \le [G: N]$ coprime to $[G: N]$ and the second sum is over proper subgroups of $G$ that contain $N$.

\begin{lemma}
\label{lem:RG_tensor}
Choose finite groups $M, N$, take $S =  M \times N$, and choose a finite cyclic extension $G/S$ of groups such that $M$ and $N$ are both normal in  $G$. Choose $\sigma \in G$ generating $G/S$.

Take $\mathscr{M}$ to be a set of normal subgroups of $G$ contained in $M$, and take $\mathscr{N}$ to be a set of normal subgroups of $G$ contained in $N$.  We may view $\mathscr{M}$ also as a set of normal subgroups of $G/N$, and $\mathscr{N}$ as a set of normal subgroups of $G/M$.

Then multiplication of class functions defines an isomorphism
\begin{equation}
\label{eq:RG_tensor}
\mathcal{R}(\sigma (S/M); \,\mathscr{N}) \otimes \mathcal{R}(\sigma(S/N); \,\mathscr{M}) \xrightarrow{\,\,\sim\,\,}\mathcal{R}\left(\sigma S; \,\mathscr{M}\cup\mathscr{N}\right)
\end{equation}

\end{lemma}
\begin{proof}

 Take 
\begin{alignat*}{2}
\mathscr{C} \, &\text{ to be the conjugacy classes of } G &&\text{contained in } \sigma S, \\
\mathscr{C}_N \, &\text{ to be the conjugacy classes of } G/N \quad&&\text{contained in } \sigma (S/N), \text{ and} \\
\mathscr{C}_M \, &\text{ to be the conjugacy classes of } G/M &&\text{contained in } \sigma (S/M). \\
\end{alignat*}
We claim that the natural map
\[\mathscr{C} \to \mathscr{C}_N \times \mathscr{C}_M\]
is a bijection. It is clearly surjective since $M$ and $N$ have trivial intersection. 

Now suppose we have chosen $\tau_1$ and $\tau_2$ in  $\sigma S$ which have equal image in $\mathscr{C}_N \times \mathscr{C}_M$. Then there is $m \in G$ so $m \tau_1 m^{-1}$ and $\tau_2$ have equal image in $G/N$. Multiplying $m$ on the right by a power of $\tau_1$ as necessary, we may assume that $m$ lies in $S$. Discarding the component in $N$, we may further assume $m$ lies in $M$. We similarly may find $n$ in $N$ so $n \tau_1n^{-1}$ and $\tau_2$ have equal image in $G/M$. Since $N \cap M = 1$, we conclude that
\[nm \tau_1 (nm)^{-1} = \tau_2,\]
establishing injectivity of this map. So multiplication of class functions defines an isomorphism
\begin{equation}
\label{eq:multiplication}
\mathcal{R}(\sigma(S/M);\,\emptyset ) \otimes \mathcal{R}(\sigma(S/N);\,\emptyset) \xrightarrow{\,\,\sim\,\,} \mathcal{R}(\sigma S;\,\emptyset ). \end{equation}

Given a subgroup $N_1$ in $\mathscr{N}$, define
\[\lambda: \mathcal{R}(\sigma(S/M);\,\emptyset ) \to \mathcal{R}(\sigma(S/M);\,\emptyset )\]
by
\[\lambda(f)(x) = \sum_{y \in N_1} f(xy)\quad\text{for all } x \in \sigma (S/M).\]
The kernel of this map is $\mathcal{R}(\sigma(S/M);\, N_1)$. Furthermore, given $v = \sum_i f_i \otimes g_i$ in the domain of the map \eqref{eq:multiplication}, we see that $v$ maps into $\mathcal{R}(\sigma S;\, N_1)$
if and only if
\[\sum_{i} \lambda(f_i)(x) \cdot g_i(x)  = 0 \quad\text{ for all } x\in \sigma S.\]
Since \eqref{eq:multiplication} is an isomorphism, this condition is equivalent to $v$ lying in the kernel of the map
\[\mathcal{R}(\sigma(S/M);\,\emptyset ) \otimes \mathcal{R}(\sigma(S/N);\,\emptyset )\xrightarrow{\,\,\lambda \,\otimes \,\text{Id}\,\,} \mathcal{R}(\sigma(S/M);\,\emptyset ) \otimes \mathcal{R}(\sigma(S/N);\,\emptyset ).\]
The kernel of this map is $\mathcal{R}(\sigma(S/M); \,N_1 ) \otimes \mathcal{R}(\sigma(S/N);\,\emptyset )$ since finite-dimensional vector spaces are flat. 

Repeating this argument for the other subgroups in $\mathscr{M}$ and $\mathscr{N}$ shows we have isomorphisms
\begin{alignat*}{3}
&\mathcal{R}(\sigma(S/M);\, \mathscr{N}) &&\otimes \mathcal{R}(\sigma(S/N);\,\emptyset) &&\xrightarrow{\,\,\sim\,\,} \mathcal{R}(\sigma S;\, \mathscr{N}) \quad\text{and}\\
&\mathcal{R}(\sigma(S/M);\, \emptyset) &&\otimes \mathcal{R}(\sigma(S/N);\,\mathscr{M}) &&\xrightarrow{\,\,\sim\,\,} \mathcal{R}(\sigma S;\, \mathscr{M}),
\end{alignat*}
that agree on the intersection of their domain. The intersection of these maps has domain and codomain equaling \eqref{eq:RG_tensor}, giving the result.
\end{proof}

\begin{lemma}[Mackey's formula]
\label{lem:Mackey}
Choose finite groups $M, N$, take $S =  M \times N$, and choose a finite extension $G/S$ of groups such that $M$ and $N$ are both normal in  $G$. Take $\mathscr{M}$ to be a set of normal subgroups of $G$ contained in $M$, and take $\mathscr{N}$ to be a set of normal subgroups of $G$ contained in $N$.

Then multiplication of class functions defines a homomorphism
\begin{equation}
\label{eq:IG_tensor}
\mathcal{I}(G/M; \,\mathscr{N}) \otimes \mathcal{I}(G/N; \,\mathscr{M}) \xrightarrow{\quad}\mathcal{I}\left(G; \,\mathscr{M}\cup \mathscr{N}\right).
\end{equation}
\end{lemma}
\begin{proof}
Choose nilpotent subgroups 
\[H_1\le G/M \quad\text{and}\quad H_2 \le G/N,\]
and take $\psi_1: H_1 \to \C$ and $\psi_2: H_2 \to \C$. We assume that
\[H_1 \cap N_0 \not \le  \ker \psi_1 \quad\text{and}\quad H_2 \cap M_0 \not \le \ker \psi_2 \quad\text{for all}\quad N_0 \in \mathscr{N}, \,\, M_0 \in \mathscr{M}.\]
To prove the lemma, it suffices to show that
\[\Ind_{H_1}^{G/M} \psi_1 \cdot \Ind_{H_2}^{G/N} \psi_2 \in \mathcal{I}(G;\, \mathscr{M}\cup \mathscr{N}).\]
Call this character $\chi$. Viewing $\psi_1$ as a character on $H_1M$ and $\psi_2$ as a character on $H_2N$, we may use Frobenius reciprocity and Mackey's formula \cite[17.3, 17.4]{Huppert98} to rewrite $\chi$ in the form
\[\Ind^{G}_{H_1M}\left( \psi_1 \cdot \res_{H_1M} \left(\Ind_{H_2N }^{G}\psi_2\right)\right) = \sum_{\tau \in B} \Ind^{G}_{H_\tau}\left(\psi_{\tau}\right),\]
where $B$ is some subset of $G$ and  where we have taken
\[H_\tau=H_1M \cap \tau^{-1} H_2N \tau\quad\text{and}\quad \psi_{\tau} = \res_{H_\tau}(\psi_1) \cdot \res_{H_\tau}(\psi_2^{\tau}),\]
 where $\psi_{2}^{\tau}$ denotes the linear character on $\tau^{-1} H_2N \tau$ defined by 
\[x \mapsto \psi_2(\tau x  \tau^{-1}).\]
Since $M \cap N = 1$, the natural projection
\[H_{\tau} \to H_1N/N \times \tau^{-1}H_2M\tau/M\]
is injective, so $H_{\tau}$ is nilpotent, and $\chi$ lies in $\mathcal{I}(G;\,\emptyset)$.

We have
\[H_{\tau} \cap S \,=\, H_1 \times \tau^{-1}H_2 \tau \,\le\, N \times M.\]
The kernel of $\res_{H_{\tau} \cap S}\, \psi_2^{\tau}$ in this group is $H_1 \times \ker \psi_2^{\tau}$, and the kernel of $\res_{H_{\tau} \cap S} \,\psi_1$ is $\ker \psi_1 \times \tau^{-1} H_2 \tau$. For any $N_0$ in $\mathscr{N}$, we see that the former kernel contains $N_0 \cap H_{\tau}$, while the latter kernel does not. So $\psi_{\tau}$ has kenrel not containing $N_0 \cap H_{\tau}$, giving
\[\chi \in \mathcal{I}(G;\, N_0).\]
Repeating this argument for the other subgroups in $\mathscr{N}$ and $\mathscr{M}$ gives the result.
\end{proof}

The following proposition will be needed to handle the abelian part of the socle of $G$ in the proof of Theorem \ref{thm:faithful}.

\begin{proposition}
\label{prop:abelian_socle}
Take $G/N$ to be a finite cyclic extension of a finite abelian group $N$, and take $\mathscr{N}$ to be a set of normal subgroups of $G$ contained in $N$. Then
\[\mathcal{R}(G;\, \mathscr{N}) = \mathcal{I}(G;\, \mathscr{N}).\]
\end{proposition}
\begin{proof}
Suppose the proposition were not true, and choose a counterexample $(G, N, \mathscr{N})$ where $|G|$ is minimized. Without loss of generality, we will assume that the center of $G$ is contained in $N$. Take $\sigma$ to be an element in $G$ generating $G/N$.

Take $\chi$ to be an irreducible character of $G$ whose kernel contains no subgroup in $\mathscr{N}$. Such $\chi$ generate $\mathcal{R}(G;\, \mathscr{N})$ by the orthogonality of irreducible characters, so we will establish a contradiction if we can prove that $\chi$ lies in $\mathcal{I}(G;\,\mathscr{N})$.

Since $N$ is abelian, its irreducible characters are all linear. Thus, since $G/N$ is cyclic, $\chi$ must be induced from a $1$-dimensional character $\psi$ of some subgroup $H$ of $G$ containing $N$ \cite[Theorem 6.22]{Isaacs76}. By induction, $\psi$ must lie in $\mathcal{I}(H;\, \mathscr{N})$ unless $H = G$. So we may assume that $\chi$ is a linear character of $G$

Take $K$ to be kernel of $\chi$ in $G$, and take $Z$ to be the center of $G$. Taking $H = K \cap Z$, we first suppose that $H$ is nontrivial. Note that the preimage of a nilpotent subgroup of $G/H$ in $G$ is nilpotent. So the projection $G \to G/H$ defines a map
\[\mathcal{I}(G/H;\, \mathscr{N}_1) \to \mathcal{I}(G;\, \mathscr{N}_2),\]
where $\mathscr{N}_1$ is the set of normal subgroups of $G/H$ not contained in $K/H$, and $\mathscr{N}_2$ is the set of normal subgroups $N$ of $G$ such that $N$ is not contained in $K$. We have $\mathscr{N}_2 \supseteq \mathscr{N}$, so the induction step shows that $\chi$ lies in $\mathcal{I}(G;\, \mathscr{N})$.

So we may assume $K \cap Z$ is trivial. Since $G/N$ is generated by $\sigma$ and $N$ is abelian, the map $\tau: N \to N$ defined by $\tau(x) =  [\sigma, x]$ fits in an exact sequence
\[0 \to Z \to N \xrightarrow{\,\,\tau\,\,} G_0 \to 0,\]
where $G_0$ is the derived subgroup of $G$. Since $K \cap Z$ is trivial, this sequence splits, so $N = Z \times G_0$. Furthermore, since any normal subgroup $M$ of $G$ contained in $N$ will also contain $\tau(M)$, we find that normal subgroups of $G$ in $N$ may be written in the form
\[Z_1 \times G_{01}\quad\text{with}\quad Z_1 \le Z \,\text{ and }\, G_{01} \le G_0.\]
If $Z_1 \times G_{01}$ lies in $\mathscr{N}$, then $Z_1$ must be nontrivial since $\chi$ lies in $\mathcal{R}(G;\,\mathscr{N})$. So, taking $\mathscr{N}_3$ to be the set of minimal subgroups of $Z$, we have
\[\mathcal{I}(G;\, \mathscr{N}_3) \subseteq \mathcal{I}(G;\,\mathscr{N}) \quad\text{and}\quad \chi \in \mathcal{R}(G;\, \mathscr{N}_3),\]
with the second claim following from $K \cap Z = 1$. So we may assume that $\mathscr{N}_3= \mathscr{N}$.

If $Z = 1$, the claim follows from Artin's induction theorem. If $|Z|$ is a prime power, it follows from Theorem \ref{thm:T(G,N)}. So suppose $|Z|$ is divisible by at least two distinct primes. Taking $p_1, \dots, p_r$ to be the disinct primes dividing $|Z|$, we take $A$ to be the subgroup of $Z$ of order $p_1$ and $B$ to be the subgroup of order $p_2 \dots p_r$. We take $\mathscr{A}$ to consist of the minimal subgroups of $A$ and $\mathscr{B}$ to consist of the minimal subgroups of $B$.  There are then integers $a, b$ with $a + b = 1$ such that $\chi^a$ is in the image of $\mathcal{R}(G/A, \mathscr{B})$ and $\chi^b$ is in the image of $\mathcal{R}(G/B, \mathscr{A})$. The claim then follows from the induction step and Lemma \ref{lem:Mackey}.
\end{proof}

\begin{proposition}
\label{prop:fake_socle}
Choose a finite group $G$, an abelian normal subgroup $A$ of $G$, and normal subgroups $N_1, \dots, N_k$ of $G$ such that the natural map
\[A \times N_1 \times \dots \times N_k \to G\]
is an injective homomorphism. 
Take $\mathscr{A}$ to be a set of normal subgroups of $G$ contained in $A$.

Then
\[\mathcal{R}(G;\, \mathscr{A}\cup\{ N_1, \dots, N_k\}) = \mathcal{I}(G;\, \mathscr{A}\cup\{N_1, \dots, N_k\}).\]
\end{proposition}
\begin{proof}
Working inductively, we may assume that the proposition has been shown for all groups of order less than $|G|$.

Take $N = A \times N_1 \times \dots \times N_k$. By Lemma \ref{lem:faith_cyclic}, we may assume that $G/N$ is cyclic. Choose $\sigma \in G$. 
From \eqref{eq:cyclic_decomp} and the induction step, we see that it suffices to prove that
\[ \mathcal{R}(\sigma N; \, \mathscr{A}\cup\{N_1, \dots, N_k\})\subseteq \mathcal{I}(G;\, \mathscr{A}\cup\{ N_1, \dots, N_k\}).\]

Take $M_0 = N_1 \times \dots \times N_k$. For $1 \le i \le k$, take 
\[M_i = A \times \prod_{j \ne i} N_j.\]
By repeatedly applying Lemma \ref{lem:RG_tensor}, we have
\[\mathcal{R}(\sigma N;\, \mathscr{A}\cup\{ N_1, \dots, N_k\}) \cong \mathcal{R}(\sigma (N/M_0);\,\mathscr{A}) \otimes  \bigotimes_{i \le k}\mathcal{R}(\sigma (N/M_i);\,N_i).\]
By Proposition \ref{prop:abelian_socle} and Theorem \ref{thm:T(G,N)}, this right hand side is a subspace of
\[\mathcal{I}(G/M_0;\,\mathscr{A}) \otimes  \bigotimes_{i \le k}\mathcal{I}(G/M_i;\,N_i).\]
By Lemma \ref{lem:Mackey}, multiplication of class functions takes this last space into
\[\mathcal{I}(G;\, \mathscr{A}\cup\{ N_1, \dots, N_k\}),\]
giving the result.
\end{proof}

\begin{proof}[Proof of Theorem \ref{thm:faithful}]
Choose a finite group $G$, and take $\mathscr{N}$ to be the collection of distinct minimal normal subgroups of $G$. We claim that
\begin{equation}
\label{eq:complex_faithful}
\mathcal{R}(G; \mathscr{N}) = \mathcal{I}(G; \,\mathscr{N}).
\end{equation}
This will imply Theorem \ref{thm:faithful}. To see this, take $\mathcal{R}_{\QQ}(G; \, \mathscr{N})$ to be the $\QQ$-linear combinations of irreducible faithful characters of $G$, and take $\mathcal{I}_{\QQ}(G; \, \mathscr{N})$ to be the $\QQ$-vector subspace of this spanned by the characters considered in Definition \ref{defn:RI}. Then we have
\[\mathcal{R}_{\QQ}(G; \, \mathscr{N}) \otimes \C \cong \mathcal{R}_{\QQ}(G; \, \mathscr{N})\quad\text{and}\quad  \mathcal{I}_{\QQ}(G; \, \mathscr{N}) \subseteq \mathcal{R}_{\QQ}(G; \, \mathscr{N}).\]
The relation \eqref{eq:complex_faithful} implies $\mathcal{R}(G; \, \mathscr{N})$ and $\mathcal{I}(G; \, \mathscr{N})$ have the same dimension as complex vector spaces. The above relations then show
\[\dim_{\QQ}\mathcal{R}_{\QQ}(G; \, \mathscr{N}) = \dim_{\QQ}\mathcal{I}_{\QQ}(G; \, \mathscr{N}),\]
implying that these vector spaces are equal and giving the theorem.

To prove \eqref{eq:complex_faithful},we consider the socle $\text{soc}(G)$ of the group $G$, which is the minimal normal subgroup of $G$ containing all minimal normal subgroups of $G$. This group takes the form
\[\text{soc}(G) \cong A \times N_1 \times \dots \times N_k,\]
where the $N_i$ are the  nonabelian minimal normal  subgroups of $G$, and where $A$ is an abelian normal subgroup containing all of the abelian minimal normal subgroups of $G$ \cite[Lemma 42.9]{Huppert98}. Taking $\mathscr{A}$ to be the set of abelian minimal normal subgroups of $G$, we see that the result follows from Proposition \ref{prop:fake_socle}.
\end{proof}

\section{Bounding smoothed character sums}
\label{sec:Lbound}

In this section, we begin by recalling the definition of Artin $L$-functions $L(s,\chi)$ and some of their basic analytic properties (e.g., the convexity bound).  The main result of this section, Proposition~\ref{prop:smoothed_character_sum}, uses these properties to give an approximation for the sum of the coefficients of $L(s,\chi)$ and related series over squarefree ideals. 

\subsection{Artin $L$-functions}

We begin by recalling the definition of Artin $L$-functions, in part to demonstrate the notation we shall use.  Let $F$ be a number field with absolute Galois group $G_F$ and degree $n$ over $\mathbb{Q}$, and let $\chi\colon G_F \to \mathbb{C}$ be a character of $G_F$ with degree $d$.  That is, there is a representation $\rho\colon G_F \to \mathrm{GL}_d(\mathbb{C})$ such that $\chi = \mathrm{tr} \rho$.  Let $K/F$ be the extension corresponding to $\chi$, i.e. the kernel field $\overline{F}^{\ker \rho}$.  For any prime $\mfp$ of $F$, let $D_\mfp$ and $I_\mfp$ denote the decomposition and inertia groups associated with a fixed prime $\mathfrak{P}$ of $K$ lying over $\mfp$.  Let $V$ be the space underlying $\rho$, and observe that $D_\mfp / I_\mfp$ acts on $V^{I_\mfp}$, the subspace of $V$ fixed by the inertia subgroup.  Letting $\sigma_\mfp$ denote the Frobenius element in $D_\mfp / I_\mfp$, we define the local Euler factor $L_\mfp(s,\chi)$ by
    \[
        L_\mfp(s,\chi)
            := \det\left( 1 - (\mathrm{N} \mfp)^{-s}\rho(\sigma_\mfp) | V^{I_\mfp} \right)^{-1}.
    \] 
    We then define $L(s,\chi)$ to be the product over prime ideals of the local factors, that is
    \[
        L(s,\chi)
            := \prod_{\mfp} L_\mfp(s,\chi).
    \]
Note that for every prime $\mfp$ of $F$, there exist ``local roots'' $\alpha_{1}(\mfp),\dots,\alpha_d(\mfp) \in \mathbb{C}$ of absolute value at most $1$ such that
    \begin{equation} \label{eqn:local-roots}
        L_\mfp(s,\chi) = \prod_{i=1}^d (1 - \alpha_i(\mfp) (\mathrm{N} \mfp)^{-s} )^{-1}.
    \end{equation}
For primes $\mfp$ that are unramified in $K$, each $\alpha_i(\mfp)$ is a root of unity and $\alpha_1(\mfp) + \dots + \alpha_d(\mfp) = \chi(\mathrm{Frob}_\mfp)$.  For primes $\mfp$ that are ramified in $K$, there is some $d_\mfp \leq d$ such that (reordering if necessary) $\alpha_1(\mfp),\dots,\alpha_{d_\mfp}(\mfp)$ are roots of unity and $\alpha_i(\mfp) = 0$ for $d_\mfp + 1 \leq i \leq d$.

With $\chi$ as above, we define the number $r = r(\chi)$ by means of the expression
    \begin{equation} \label{eqn:order-pole}
        r := \langle \chi, \mathbf{1} \rangle_G = \frac{1}{|G|} \sum_{g \in G} \chi(g),
    \end{equation}
where $G = \mathrm{Gal}(K/F)$ and we regard $\chi$ as a character of $G$.  Since $\chi$ is a character and $r$ is the multiplicity of the trivial representation inside the representation $\rho$ associated with $\chi$, we have that $r$ is an integer satisfying $0 \leq r \leq d$ and that $r = \mathrm{ord}_{s=1} L(s,\chi)$.

Let $\mathfrak{f}_\chi$ be the Artin conductor associated with $\chi$, for example following \cite[p. 533]{Neuk99}, and define the quantity $Q_\chi$ by $Q_\chi := \Delta_F^d \mathrm{N} \mathfrak{f}_\chi$.  So doing, the Artin $L$-function $L(s,\chi)$ satisfies a functional equation of the form
\begin{equation}\label{eqn:artin-fe}
    L(s, \chi) 
        = w \cdot \pi^{n d (s-1) } \cdot Q_{\chi}^{\frac{1 - 2s}{2}}\cdot  2^{nd s} \cdot \Gamma(1-s)^{n d} \cdot \sin\left(\frac{\pi s}{2}\right)^{r_1}\cdot \cos\left(\frac{\pi s}{2}\right)^{r_2} \cdot L\left(1-s, \overline{\chi}\right)
\end{equation}
where $r_1,$ and $r_2$ are non-negative integers satisfying $r_1+r_2=nd$ and where $w \in \mathbb{C}$ has modulus $1$ \cite[Theorem 12.6]{Neuk99}.

We now record an explicit form of the convexity bound for $L(s,\chi)$. We will take the notation
\[Q_{\chi}(t) = Q_{\chi} \cdot (1 + |t|)^{nd}.\]

\begin{lemma}
\label{lem:Lind}
    Let $L(s,\chi)$ be an Artin $L$-function of degree $d$ over a number field $F$ of degree $n$, and suppose that $(s-1)^r L(s, \chi)$ is entire for some $0 \leq r \leq d$.
    There for all $0 < \delta <1/2$ 
    and all complex $s = \sigma + it$ with $- \delta \le \sigma \le 1 + \delta$, we have
    \[
        \left|\left(\frac{s-1}{s+1}\right)^r L(s, \chi) \right|
            \leq 3^{2d} e^{nd} \delta^{-d} \cdot  Q_{\chi}(t)^{\frac{1+\delta-\sigma}{2}} \cdot \left( 3 + \frac{\log \Delta_F}{2n}\right)^{nd}.
    \]

\end{lemma}

\begin{proof}
    First, for any positive $\delta < 1/2$ and any $t \in \mathbb{R}$, we have that
        \[
            \left|L( 1 + \delta + it, \chi) \right|
                \le \zeta_F(1 + \delta)^{d},
        \]
    where $\zeta_F(s)$ is the Dedekind zeta function of $F$, as follows from \eqref{eqn:local-roots}.
    The same bound holds for $L(1+\delta - it, \overline{\chi})$, which implies via the functional equation \eqref{eqn:artin-fe} that
        \begin{align*}
            |L(-\delta + it, \chi)|
                \leq \pi^{-nd(1+\delta)} \cdot 2^{-nd\delta} \cdot Q_{\chi}^{\frac{1+2\delta}{2}} \cdot |\Gamma(1+\delta+it)|^{nd} e^{\frac{\pi nd |t|}{2}}\cdot  \zeta_F(1+\delta)^d.
        \end{align*}

    Applying the explicit error estimate  for Stirling's approximation \cite[(3.11)]{Boyd94} in the case $N = 1$ at $s = 1 + \delta + it$ gives
    \[\left|\Gamma(s) \right| \le \sqrt{2\pi} \cdot \left|s^{s - \frac{1}{2}}\right| \cdot e^{-1 - \delta} \cdot \left(1 + \frac{6}{(2 \pi)^2}\right).\]
    Note that
    \[\left|s^{s - \frac{1}{2}}\right| = |s|^{\frac{1}{2} + \delta}\cdot  e^{-\text{arg}(s) \cdot t}, \]
    where the argument $\text{arg}(s)$ lies in $(-\pi/2, \pi/2)$. Assuming this argument is nonnegative, we then have
    \[t^{-1}(1 + \delta) = \tan\left(\tfrac{1}{2}\pi - \text{arg}(s)\right) \ge \tfrac{1}{2}\pi - \text{arg}(s).\]
    Together with a symmetric argument when $\text{arg}(s)$ is negative, we may conclude
    \[ \left|s^{s - \frac{1}{2}}\right| \le |s|^{\frac{1}{2} + \delta}\cdot \exp\left(1 + \delta - \tfrac{1}{2}\pi |t|\right), \]
    and a computation gives
    \[\left|\Gamma(1 + \delta + it)\right| \le  \pi \cdot (1 + \delta + |t|)^{\frac{1}{2} + \delta} \cdot \exp\left(-\tfrac{1}{2} \pi |t|\right) \le  \pi \cdot (1 + \delta)  \cdot (1 +  |t|)^{\frac{1}{2} + \delta} \cdot \exp\left(-\tfrac{1}{2} \pi |t|\right)\]
So
        \[
            |L(-\delta + it, \chi)|
                \leq Q_\chi(t)^{\frac{1+2\delta}{2}} \cdot \zeta_F(1+\delta)^d.
        \]
    Thus, by the Phragmen--Lindel\"of convexity principle \cite[Theorem 5.53]{Iwan04}, for any $\sigma$ satisfying $-\delta < \sigma < 1+\delta$ and $s = \sigma + it$, we find
        \begin{equation} \label{eqn:convexity-with-zeta}
           \left| \left(\frac{s-1}{s+1}\right)^r L(s, \chi) \right| \leq 3^{\frac{r(1+\delta-\sigma)}{1+2\delta}} \cdot
                 Q_\chi(t) ^{\frac{1+\delta-\sigma}{2}} \cdot  \zeta_F(1+\delta)^d.
        \end{equation}

It remains to bound $\zeta_F(1 + \delta)$. We claim that
\begin{equation}
    \label{eq:zetaF_bnd}
    \left|  \zeta_F(1 + \delta)\right| \le \delta^{-1} \cdot 2\cdot 3^{1/4} \cdot e^{n} \cdot \max\left(1 + \frac{\log \Delta_F}{2n}, 3\right)^n \quad\text{ for } \delta \in (0, 1/2).
\end{equation}
If $F = \QQ$, this is clear from the relationship $\zeta(1 + \delta) \le (1 + \delta^{-1})$. 
Otherwise, take $\chi_0$ to be the trivial character on $G_F$. Applying \eqref{eqn:convexity-with-zeta} with $\delta_0 = \min(\frac{1}{2}, \frac{2n}{\log \Delta_F})$ gives
\[\zeta_F(1 + \delta) \le \delta^{-1} \cdot 2 \cdot 3^{1/4} \cdot e^n  \cdot \zeta_F(1 + \delta_0) \quad\text{ for } \delta \in (0, \delta_0).\]
For $\delta \in (\delta_0, 1/2)$, we instead use the inequality
\[\zeta_F(1 + \delta) < \zeta_F(1 + \delta_0) < \delta^{-1} \cdot \zeta_F(1 + \delta_0).\]
In either case, we fnd that \eqref{eq:zetaF_bnd} follows from the inequality
\[\zeta_F(1 + \delta_0) \le \zeta(1 + \delta_0)^n \le (1 + \delta_0^{-1})^n = \max\left(1 + \frac{\log \Delta_F}{2n}, 3\right)^n.\]
 Combining \eqref{eq:zetaF_bnd} with \eqref{eqn:convexity-with-zeta} then gives the lemma.
\end{proof}

\begin{rmk}
    The hypothesis that $(s-1)^r L(s,\chi)$ is entire is expected to hold for all characters $\chi$, but at present this is known only for characters expressible as a sums of monomial characters by class field theory and for a few scattered other classes of $\chi$ that will not be relevant for our purposes. Here, a \emph{monomial character} is defined as a character induced from a one dimensional character of some open subgroup.
\end{rmk}

\subsection{Acceptable multiplicative functions}
An Artin $L$-function $L(s, \chi)$ defined over a number field $F$ may be written in the form $\sum_{\mfa} \frac{f(\mfa)}{N\mfa^s}$, where $f$ is a multiplicative function on the integral ideals of $F$. If $f_1$, $f_2$ are the multiplicative functions defined this way from characters $\chi_1$, $\chi_2$, and if $f$ is the multiplicative function defined from the product character $\chi_1 \cdot \chi_2$, we find that
\begin{equation}
\label{eq:ff1f2}
f(\mfa) = f_1(\mfa) \cdot f_2(\mfa)
\end{equation}
for all squarefree integral ideals $\mfa$ of $F$ that are divisible by no prime where both $\chi_1$ and $\chi_2$ ramify.

For our applications, we would like \eqref{eq:ff1f2} to hold for all integral ideals. This requires us to modify our approach for defining a multiplicative function from a Galois character.

\begin{defn}
    Let $F$ be a number field, let $\chi\colon G_F \to \C$ be a character of degree $d$, and let $K/F$ be a Galois extension containing the kernel field of $\chi$. We will assume that $L(s, \chi)$ is entire except potentially for a pole at $s = 1$.
    
    An \emph{acceptable multiplicative function} associated with the tuple $(K/F, \chi, S)$ is a multiplicative function $f$ on the ideals $\mathfrak{a}$ of $F$ that is supported on squarefree ideals, and which satisfies $f(\mfp) = \chi(\mathrm{Frob}_\mfp)$ for primes $\mfp \not\in S$ and $|f(\mfp)| \leq d$ for primes $\mfp \in S$.

    Given an acceptable multiplicative function $f$, we let $L(s,f)$ denote its Dirichlet series, that is,
    \[
        L(s,f) 
            := \sum_{\mathfrak{a}} \frac{ f(\mathfrak{a})}{(\mathrm{N} \mathfrak{a})^s} = \prod_{\mfp} \left( 1 + \frac{f(\mfp)}{\mathrm{N}\mfp^s}\right),
    \]
    which is absolutely convergent in the region $\Re(s)>1$.
\end{defn}

We need to show that $L(s, f)$ is not too much larger than $L(s, \chi)$.
\begin{lemma}
\label{lem:Lfs}
Let $f$ be an acceptable multiplicative function associated with a tuple $(K/F, \chi, S)$.  Then for any $s = \sigma + it$ with $\sigma \ge \sigma_0$ for some $\sigma_0 \in (1/2, 1)$,
we have
\[\log\left|\frac{L(s, f)}{L(s,\chi)}\right| \le  \frac{3 n \cdot d^{4 - 4\sigma}}{ 2\sigma - 1} + 2dn\cdot\frac{ (2d^2 + \# S)^{1 - \sigma_0} - 1}{1 - \sigma_0} \]
\end{lemma}
\begin{proof}
For any prime $\mfp$ of $F$, take
    \[
        a_{\mfp} = \log|1 + f(\mfp) \cdot (\mathrm{N} \mfp)^{-s}| - \log |L_{\mfp}(s,\chi)|,
    \]
where $L_{\mfp}$ denotes the Euler factor of $L(s,\chi)$ at $\mfp$. 
By \eqref{eqn:local-roots} and the surrounding discussion, we may write the second term above in the form
    \[
        - \log L_{\mfp}(\chi, s) 
            = \log(1 - \alpha_{1}(\mfp) (\mathrm{N} \mfp)^{-s}) + \dots + \log (1 - \alpha_{d_{\mfp}}( \mfp) (\mathrm{N} \mfp)^{-s}),
    \]
where $d_{\mfp} \le d$ and the $\alpha_{i}(\mfp)$ are all roots of unity. From these formulae
and the simple inequality $\log |1+x| \leq |x|$
we may conclude
    \begin{equation} \label{eqn:basic-ap-inequality}
        a_{\mfp}  \le 2d (\mathrm{N}\mfp)^{-\sigma}
    \end{equation}
for all primes $\mfp$. If $\mfp$ is outside $S$ and $\mathrm{N} \mfp \ge 2d^2$, we have a stronger inequality: the Taylor series for $\log(1+x)$ gives
    \begin{equation} \label{eqn:better-ap-inequality}
        a_{\mfp} 
            \le \sum_{k \ge 2}\frac{ d^k + d}{k} (\mathrm{N} \mfp)^{-k\sigma} 
            \le d^2 \cdot (\mathrm{N} \mfp)^{-2\sigma} \cdot \sum_{k \ge 2} \frac{2^{2-k/2}}{k} 
            < 3d^2 \cdot (\mathrm{N} \mfp)^{-2\sigma}.
    \end{equation}
    
Now, given an integer $m > 1$, there are at most $n$ primes in $F$ with norm $m$.  Applying this fact with \eqref{eqn:basic-ap-inequality}, we find
    \[
        \sum_{\mfp \in S} a_{\mfp} + \sum_{\mathrm{N} \mfp \le 2d^2} a_{\mfp} 
            \le 2dn \cdot \int_1^{2d^2 + \# S}x^{-\sigma }dx  
            \le 2dn \cdot\frac{ (2d^2 + \# S)^{1 - \sigma_0} - 1}{1 - \sigma_0},
\]
while by invoking \eqref{eqn:better-ap-inequality}, we find
    \[
        \sum_{\substack{\mfp \not \in S \\ \mathrm{N} \mfp > 2d^2}} a_{\mfp} 
            \le 3d^2 n \int_{2d^2}^{\infty} x^{-2\sigma} dx 
            \le \frac{3 n \cdot d^{4 - 4\sigma}}{ 2\sigma - 1}.
    \]
Since 
    \[
        \log \left| \frac{L(s,f)}{L(s,\chi)} \right|
            = \sum_{\mfp} a_{\mfp},
    \]
the lemma follows.
\end{proof}

We will apply Lemma~\ref{lem:Lfs} in two special cases.

\begin{lemma}
\label{lem:L_line_circle}
Choose $\delta$ in $(0, 1/4]$. If we take $s = 1/2 + \delta + it$ with $t$ real, we have
\[\left| L(s, f)\right| \le  \left(\log e^3 Q_{\chi}(t)\right)^d \cdot \left(\log e^3 \Delta_F\right)^{nd} \cdot Q_{\chi}(t)^{\frac{1 - 2\delta}{4}} \cdot \exp\left(6nd + 3nd^2 \delta^{-1} + 4dn (\#S)^{1/2}\right).\]
If $s$ is instead a complex number satisfying $|s - 1| = \delta$, we have
\[\left| L(s, f)\right| \le \delta^{-d -r} \cdot Q_{\chi}^{\delta}  \cdot \left(\log e^3 \Delta_F\right)^{nd} \cdot \exp\left(11nd + 2dn \cdot \frac{(2d^2 + \# S)^{\delta} -1 }{\delta}\right).\]

\end{lemma}

\begin{proof}
Start with $s = 1/2 + \delta + it$ as  in the first part. Applying Lemma \ref{lem:Lind} with $(\delta, \sigma)$ set to
\[\left(\left(\log e^3 Q_{\chi}(t)\right)^{-1},\,\, 1/2 + \delta\right)\]
gives
\begin{align*}
\left|L(s, \chi)\right| &\le 7^r \cdot 3^{2d} \cdot  e^{nd} \cdot \left(\log e^3 Q_\chi(t)\right)^d \cdot e^{1/2}\cdot Q_{\chi}(t)^{\frac{1 - 2\delta}{4}} \cdot \left(3 + \frac{\log \Delta_F}{2n}\right)^{nd} \\
&\le e^{6nd} \cdot \left(\log e^3 Q_{\chi}(t)\right)^d \cdot \left(\log e^3 \Delta_F\right)^{nd} \cdot Q_{\chi}(t)^{\frac{1 - 2\delta}{4}} . 
\end{align*}
Applying Lemma \ref{lem:Lfs} gives
\[\log\left|\frac{L(s, f)}{L(s,\chi)}\right| \le \frac{3nd^2}{2\delta} + 2dn \frac{(2d^2 + \#S)^{1/2} - 1}{1/2} \le \left(\frac{3}{2} + \sqrt{2}\right) nd^2\delta^{-1} + 4dn (\#S)^{1/2},\]
and the first part of the lemma follows.

For the second part, applying Lemma \ref{lem:Lind} with $(\delta, \sigma)$ set to $(\delta, \Re(s))$ gives
\begin{align*}
\left|L(s, \chi)\right| &\le \left(\frac{9}{4}\right)^{r} \cdot 3^{2d} \cdot e^{nd} \cdot \delta^{-d -r} \cdot Q_{\chi}(\Im(s))^{\delta} \cdot \left(\log e^3 \Delta_F\right)^{nd}\\
&\le \delta^{-d -r} \cdot Q_{\chi}^{\delta} \cdot e^{5nd} \cdot \left(\log e^3 \Delta_F\right)^{nd},
\end{align*}
and the lemma follows since Lemma \ref{lem:Lfs} gives
\[\log\left|\frac{L(s, f)}{L(s,\chi)}\right| \le 6nd + 2dn \cdot \frac{(2d^2 + \# S)^{\delta} - 1}{\delta}.\]
\end{proof}

\subsection{Smoothed character sums}
Our bounds for bilinear character sums are proved using bounds on smoothed character sums, which we prove from bounds on $L$-functions in the critical strip in a manner similar to Weiss \cite[Lemma 3.5]{Weiss83}. This work requires a smoothing function. The specific choice of this function will not matter much, so we choose one whose Fourier transform is easy to calculate.

\begin{defn}
Given  $H \ge 3$, we will define a holomorphic function $\eta_H$ by the formula
\[\eta_H(z) = \int_{-\log H}^{\log H} e^{1 -(z - t)^2} dt.\]
In other words, $\eta_H$ is the convolution of the indicator function on $[- \log H, \log H]$ with $e^{1-z^2}$. The Fourier transform of this function has the explicit form
\[\widehat{\eta_H}(s) = \int_\R \eta_H(x)e^{-isx} dx = 2e \sqrt{\pi} \cdot \frac{\sin (s \cdot \log H)}{s}\cdot  e^{-s^2/4}.\]
For $x$ in $[- \log H, \log H]$, we have the (crude) lower bound
\begin{equation}
    \label{eq:etaH_bnd}
    \eta_H(x) \ge \int_{0}^1 e^{1 -t^2}dt \ge 1.
\end{equation}
\end{defn}

Building on Lemma \ref{lem:L_line_circle}, the following two lemmas collect the bounds we need involving $\widehat{\eta_H}(-is) \cdot L(s, f)$ in the proof of  Proposition \ref{prop:smoothed_character_sum}.

\begin{notat}
\label{notat:smoothed}
We take $f$ to be an acceptable multiplicative function associated to the tuple $(K/F, \chi, S)$. We take $d$ to be a positive integer no smaller than degree of $\chi$ and we take $Q$ to be a real number no smaller than $Q_{\chi}$. We will assume that
\[Q \ge \max\left(2^{n}, \exp\left(4d^{1/2}\right), \exp\left(\tfrac{1}{2}d^{1/2} |S|\right)\right).\]
As above, we take $n$ to be the degree of $F$ over $\QQ$, and we will write $r$ for the degree of the pole for $L(s, \chi)$ at $s = 1$, with $r = 0$ if there is no pole.
\end{notat}

\begin{lemma}
\label{lem:Z_int}
Take all notation as above, and suppose $r \ge 1$. Take $H \ge e^4$. Then
\[\left|\frac{1}{2\pi}\oint_Z \widehat{\eta_H}(-is) \cdot L(s, f) ds\right| \le H (\log H)^{r-1} \cdot (\log Q)^{50nd},\]
where $Z$ is the counterclockwise circular path centered at $s = 1$ of radius $1/4$.
\end{lemma}

\begin{proof}
Given $\delta$ in $(0, 1/4]$ and $s \in \C$ satisfying $|s - 1 | \le \delta$, we have
\[\left|\widehat{\eta_H}(-is)\right| \le  2e\sqrt{\pi} \cdot H^{1 + \delta}\cdot \frac{e^{(1 + \delta)^2/4}}{1 - \delta} \le 20 H^{1 + \delta},\]
so Cauchy's integral formula gives
\[\left|\frac{1}{k!}\frac{d^k}{ds^k} \widehat{\eta_H}(-is)\Big|_{s = 1} \right| \le  \delta^{ - k} \cdot 20 H^{1 + \delta}\quad\text{for } k \ge 0.\]
Taking $\delta = (\log  H)^{-1}$ gives
\[\left|\frac{1}{k!}\frac{d^k}{ds^k} \widehat{\eta_H}(-is)\Big|_{s = 1}\right|  \le  H \cdot 20e  \cdot  (\log H)^{k} \le 55 H (\log H)^k.\]
We next will apply Lemma \ref{lem:L_line_circle} with
\[\delta = (2d^2 + \# S)^{-1} \cdot (\log Q)^{-1},\]
which gives
\[\left|L(s, f)\right| \le (\log Q)^{4nd} \cdot \exp\left(12nd + 2dn \frac{(2d^2 + \#S)^{\delta} - 1}{\delta}\right) \cdot (2d^2 + \# S)^{2d}\]
for $s$ satisfying $|s-1| = \delta$. Applying the mean value theorem to the function $g(x) = (2d^2 + \# S)^x$, we find that
\[\frac{(2d^2 + \# S)^{\delta} - 1}{\delta} \le g'(\delta) < e \cdot \log (2d^2 + \#S).\]
So, on this circle, we have
\[\left|L(s, f)\right| \le (\log Q)^{4nd} \cdot e^{12nd} \cdot (2d^2 + \# S)^{(2e + 2)nd} \le (\log Q)^{46nd},\]
and it follows that
\[\left|\res_{s = 1} (s-1)^k L(s, f) \right| \le (\log Q)^{46nd}\]
for $0 \le k \le r - 1$. From the residue theorem, we thus have
\begin{align*}
\left|\frac{1}{2\pi}\oint_Z \widehat{\eta_H}(-is) \cdot L(s, f) ds\right| &\le  \sum_{k = 0}^{r-1} 55 H (\log H)^k \cdot (\log Q)^{46nd} \\
&\le H (\log H)^{r-1} \cdot (\log Q)^{50nd}.
\end{align*}

\end{proof}

\begin{lemma}
\label{lem:straight}
Take all notation as in Notation \ref{notat:smoothed}, and choose $H \ge Q^{1/2}$ such that
\[16nd^2 \le \log Q^{-1/2}H.\]
Then, for any $\delta$ in $(0, 1/4]$, we have
\begin{align*}
&\left|\frac{1}{2\pi}\int_{1/2 + \delta - i\infty}^{1/2 + \delta + i\infty} \widehat{\eta_H}(-is) \cdot L(s, f) ds\right| \\
&\le\,  H \cdot (Q^{-1/2}H)^{-1/2} \cdot (\log Q)^{26nd} \cdot \exp\left(4 n^{1/2}d \sqrt{\log Q^{-1/2}H} + 4 \sqrt{2}nd^{3/4} \sqrt{\log Q}\right).
\end{align*}

\end{lemma}

\begin{proof}
We start by estimating
\begin{equation}
\label{eq:half_int}
\int_{0}^{ \infty} e^{-t^2/4} \cdot Q_\chi(t)^{\frac{1 - 2\delta}{4}} \cdot (\log e^3 Q_{\chi}(t))^d dt. 
\end{equation}
Writing the integrand as $F(t)$, we have
\[\frac{d}{dt} \log F(t) = \frac{-t}{2} + \frac{(1 -2 \delta)nd}{4(1 + t)} + \frac{d^2n}{ (1 + t) \cdot\log e^3 Q_{\chi}(t)}.\]
So long as $t \ge e - 1$, this is at most
\[\frac{-t}{2} + \frac{nd}{4 t} + \frac{d}{t}, \]
which is no greater than $-1$ if $t \ge 3\sqrt{nd}$. Since $Q_{\chi}$ is increasing with $t$ and $e^{-t^2/4}$ is no greater than 1, \eqref{eq:half_int} is at most
\begin{align*}
&\left(3\sqrt{nd} + 1\right) \cdot Q_{\chi}\left(3\sqrt{nd}\right)^{\frac{1 - 2\delta}{4}} \cdot \left(\log e^3 Q_{\chi}\left(3\sqrt{nd}\right)\right)^d \\
&\le\, Q_{\chi}^{\frac{1 - 2\delta}{4}} \cdot \left(4\sqrt{nd}\right)^{\frac{nd}{4} + 1} \cdot \left( nd \log \left(4 \sqrt{nd}\right) + \log e^3 Q_{\chi}\right)^d
\,\le\, Q^{\frac{1 - 2\delta}{4}} \cdot (\log Q)^{16nd},
\end{align*}
where we have used the fact that $\log Q$ is at least $4$ and $(\log Q)^3$ is at least $\max(n, d)$.

By Lemma \ref{lem:L_line_circle}, the integral of the lemma is then bounded by 
\[Q^{\frac{1 - 2\delta}{4}} \cdot (\log Q)^{16nd} \cdot \frac{2 \cdot 5e\sqrt{\pi}}{2\pi}H^{1/2 + \delta} \cdot \left(\log e^3 \Delta_F\right)^{nd} \cdot \exp\left(6nd + 3nd^2 \delta^{-1} + 4dn (\#S)^{1/2}\right),\]
which is at most
\[H \cdot (\log Q)^{26nd}  \cdot(Q^{-1/2}H)^{-1/2 + \delta} \cdot  \exp\left(3nd^2 \delta^{-1} + 4nd \cdot(\# S)^{1/2}\right).\]
By Cauchy's residue theorem, we may shift $\delta$ to any value in the interval $(0, 1/4]$ without changing the value of the integral. We will take
\[\delta =  n^{1/2} d (\log Q^{-1/2} H)^{-1/2};\]
note that this is at most $1/4$ by the conditions of the lemma. The result follows.
\end{proof}

\begin{proposition}
\label{prop:smoothed_character_sum}
Take all notation as in Notation \ref{notat:smoothed}, and choose $H \ge Q^{1/2}$ satisfying the condition of Lemma \ref{lem:straight}. We then have
\begin{align*}
& \left|\sum_{\mfa} f(\mfa) \cdot \eta_H(\log N\mfa)\right| \\
&\quad\le \kappa \cdot H (\log H)^{r- 1} \cdot (\log Q)^{50nd} \\
& \quad+  H \cdot (Q^{-1/2}H)^{-1/2} \cdot (\log Q)^{26nd} \cdot \exp\left(4 n^{1/2}d \sqrt{\log Q^{-1/2}H} + 4 \sqrt{2} nd^{3/4} \sqrt{\log Q}\right),
\end{align*}
where $\kappa = 1$ if $r$ is positive and $\kappa = 0$ otherwise.
\end{proposition}
\begin{proof}
Applying the inverse Fourier transform and Cauchy's integral theorem gives
\[\eta_H(x) = \frac{1}{2\pi i} \int_{2-i \infty}^{2 + i \infty} \widehat{\eta_H}(-is)e^{sx} ds\]
for any real $x$. So
\begin{equation}
\label{eq:switchable}
\sum_{\mfa} f(\mfa) \cdot \eta_H\left(\log N\mfa\right) = \sum_{\mfa} \frac{1}{2\pi i} \int_{2-i \infty}^{2 + i \infty} \widehat{\eta_H}(-is)N\mfa^{-s}ds.
\end{equation}
There is some $C > 0$ such that, for all positive $H_0$, we have
\begin{align*}
&\left|\sum_{N\mfa \ge H_0} \frac{1}{2\pi i}\int_{2-i \infty}^{2 + i \infty} \widehat{\eta_H}(-is)f(\mfa)N\mfa^{-s}ds\right| \le CH_0^{-1/2} \quad\text{and}\\
&\left|\frac{1}{2\pi i}\int_{2-i \infty}^{2 + i \infty} \sum_{N\mfa \ge H_0}  \widehat{\eta_H}(-is)f(\mfa)N\mfa^{-s}ds\right| \le CH_0^{-1/2}. 
\end{align*}
This allows us to swap the order of the sum and integral in \eqref{eq:switchable}, so
\[\sum_{\mfa} f(\mfa) \cdot \eta_H\left(\log N\mfa\right) =  \frac{1}{2\pi i} \int_{2-i \infty}^{2 + i \infty} \widehat{\eta_H}(-is)L(s, f)ds.\]
But now Cauchy's theorem gives
\[\int_{2-i \infty}^{2 + i \infty} \widehat{\eta_H}(-is)L(s, f)ds \,=\, \oint_Z \widehat{\eta_H}(-is)  L(s, f) ds \,+\, \int_{1/2 + \delta - i\infty}^{1/2 + \delta + i\infty}\widehat{\eta_H}(-is)  L(s, f) ds,\]
where $Z$ is the path from Lemma \ref{lem:Z_int} and $\delta$ is taken from $(0, 1/4]$. The result follows from Lemmas \ref{lem:Z_int} and \ref{lem:straight}.
\end{proof}

\section{Bilinear character sums}
\label{sec:Bibound}

\subsection{The theorem for large $H$}
\begin{defn}
Given a Galois extension of number fields $K_1/F$ and a nontrivial character $\chi_1$ of $\Gal(K_1/F)$, we say that $\chi_1$ is a \emph{monomial positive} if $\chi_1$ is a positive rational combination of monomial characters on $\Gal(K_1/F)$, i.e. characters induced from a linear character of some subgroup.

Given another Galois extension $K_2$ of $F$ and a choice of character $\chi_2 : \Gal(K_2/F) \to \C$, we define an inner product via the standard formula
\[\langle \chi_1, \chi_2 \rangle = \frac{1}{|G|}\sum_{\sigma \in G} \chi_1(\sigma)\overline{\chi_2(\sigma)} \quad\text{with }\, G = \Gal(K_1K_2/F).\]
\end{defn}

\begin{theorem}
\label{thm:bilinear_basic}
Choose a number field $F$ of degree $n$, and fix a positive integer $d$.

Choose integers $Q, M > 100$, and choose $M$ acceptable multiplicative functions $f_1, \dots, f_M$ over $F$. Take $(K_i/F, \chi_i, S_i)$ to be a tuple associated to $f_i$ for $i \le M$. For each $i \le M$, we will assume that $\chi_i$ is monomial positive, has degree at most $d$, and satisfies $Q_{\chi_i} \le Q$. We will also assume that $S_i$ has cardinality at most $\log Q/\log 2$.

Take $E$ to be the set of pairs $(i, j)$ with $i, j \le M$ such that $\langle \chi_i, \chi_j \rangle \ne 0$. Take $r$ to be the maximum value attained by this pairing as $(i, j)$ varies. 

Then, for $H \ge Q^d e^{16nd^4}$, we have
\begin{align*}
\sum_{N\mfa < H} \left| \sum_{i = 1}^M a_if_i(\mfa)\right|^2 &\le\,   H \cdot (\log H)^{r-1} \cdot (2d \log Q)^{50d^2n} \cdot \left(\sum_{(i, j) \in E} |a_ia_j|\right) \\
&+\, AH \cdot (Q^{-d}H)^{-1/2} \cdot \left(\sum_{i \le M} |a_i|\right)^2,
\end{align*}
where the sum on the left is over all integral ideals of $F$ of norm at most $H$, and where
\[A = (2d \log Q)^{26d^2n} \cdot \exp\left(4n^{1/2} d^2 \sqrt{\log Q^{-d}H} + 8nd^2 \sqrt{\log Q}\right). \]

\end{theorem}

We will need the following basic lemma.
\begin{lemma}
\label{lem:bilinear_lemma}
Take $f_1, \dots, f_M$ and their associated tuples as in Theorem \ref{thm:bilinear_basic}. Given $i, j \le M$, the function $f_i \cdot \overline{f_j}$ is an acceptable multiplicative function with associated character $\chi = \chi_i \cdot \overline{\chi_j}$ and associated set of places contained in $S_i \cup S_j$. Furthermore, $\chi$ is monomial positive and satisfies
\[Q_{\chi} \le Q^{2d}.\]
\end{lemma}
\begin{proof}
From the theory of products of characters, the function $\chi_i \cdot \overline{\chi_j}$ is a character. Since $\chi_i$ and $\chi_j$ are monomial positive, $\overline{\chi_j}$ is monomial positive, and Mackey's formula gives that $\chi_i \cdot \overline{\chi_j}$ is monomial positive. The claims about $f_i \cdot \overline{f_j}$ are then clear. The bound on $Q_{\chi}$ then follows from \cite[Lemma 6.6]{LOTZ}.
\end{proof}
\begin{proof}[Proof of Theorem \ref{thm:bilinear_basic}]
From \eqref{eq:etaH_bnd}, we have
\begin{align}\nonumber
\sum_{N\mfa < H} \left| \sum_{i = 1}^M a_if_i(\mfa)\right|^2 &\le \sum_{\mfa} \eta_H\left(\log N\mfa\right) \left| \sum_{i = 1}^M a_if_i(\mfa)\right|^2 \\
&= \sum_{i, j \le M}a_i \overline{a_j} \sum_{\mfa} \eta_H\left(\log N\mfa\right) \cdot (f_i \cdot \overline{f_j})(\mfa).\label{eq:first_doublesum_switch}
\end{align}
 By the Minkowski bound and the assumption $Q \ge 100$, it follows that $Q$ is at least $2^n$. Furthermore, for $i, j \le M$
\[\# (S_i \cup S_j) \le 2 \log Q/\log 2 \le 2d^{-1} \log Q^{2d}.\]
Finally, $\chi_i \cdot \overline{\chi_j}$ has degree at most $d^2$. Applying Lemma \ref{lem:bilinear_lemma}, we find that we may apply Proposition \ref{prop:smoothed_character_sum} with $(d, Q)$ set to $(d^2, Q^{2d})$ to show that \eqref{eq:first_doublesum_switch} is at most
\[\sum_{i, j \le M} |a_ia_j| \left(\kappa_{ij} \cdot H(\log H)^{r-1} \cdot (2d \log Q)^{50nd^2} + AH \cdot (Q^{-d}H)^{-1/2} \right),\]
where $\kappa_{ij} =1$ if $(i, j)$ lies in $E$ and is otherwise $0$. The result follows.
\end{proof}

\begin{rmk}
\label{rmk:min_H}
Suppose $H \le Q^d e^{16nd^4}$, so the theorem does not apply as written. In this case, we will still need to bound this bilinear sum, but a trivial bound will suffice. Specifically, we will choose $\delta$ in $(0, 1/2)$ and use
\[\sum_{N\mfa < H} \left| \sum_{i = 1}^M a_if_i(\mfa)\right|^2 \le \sum_{i, j \le M} |a_i|\cdot  |a_j|\cdot  H^{1 + \delta} \zeta_F(1 + \delta)^{d^2} \le \left(\sum_{i \le M} |a_i|\right)^2 \cdot H^{1 + \delta} (1 + \delta^{-1})^{nd^2}.\]
Taking $\delta = (\log e^3 H)^{-1}$ then gives
\begin{equation}
\label{eq:trivial_bilinear_bound}
\sum_{N\mfa < H} \left| \sum_{i = 1}^M a_if_i(\mfa)\right|^2 \le (\log e^4 H)^{nd^2 + 1}\cdot H \cdot \left(\sum_{i \le M} |a_i|\right)^2 
\end{equation}

The minimal choice of $H$ obeying the conditions of the theorem is $Q^d e^{16nd^4}$. In this case, the theorem gives
\begin{align*}
\sum_{N\mfa < H} \left| \sum_{i = 1}^M a_if_i(\mfa)\right|^2 &\le\,  Q^d \cdot (2d \log Q)^{54d^2n} \exp\left(16nd^4 + 8nd^2 \sqrt{\log Q}\right) \cdot \left(\sum_{i \le M} |a_i|\right)^2.
\end{align*}
\end{rmk}

\subsection{Applying H\"{o}lder's inequality for small $H$}
We can adapt Theorem \ref{thm:bilinear_basic} to handle smaller $H$ by applying H\"{o}lder's inequality. This approach was first used in work of Friedlander and Iwaniec \cite[(21.9)]{Fried98}. As in that original application, this method is something of a shortcut around the deeper consideration of the involved $L$-functions as appears in \cite{LOTZ}.

We start with a weak estimate for the number of squarefull ideals.

\begin{lemma}
\label{lem:squarefull}
Take $F$ to be a number field of degree $n$. Given $H \ge 100$, we have
\[\sum_{\substack{\mfr \text{ sqfull}\\ N\mfr \le H}} N\mfr^{-1/2} \le (\log e^3 \Delta_F)^{4n} \cdot \log H,\]
where the sum is over squarefull integral ideals of $F$ of norm at most $H$. We also have
\[\sum_{\mfr \text{ sqfull}} N\mfr^{-1} \le \zeta(2)^n\cdot \zeta(3)^n \le 2^n\]
\end{lemma}
\begin{proof}
For any $\delta \in (0, 1/2)$, the sum being estimated is at most
\[H^{\delta} \cdot \sum_{\substack{\mfr \text{ sqfull}}} N\mfr^{-1/2 - \delta}.\]
Since every squarefull ideal may be written in the form $\mfa^2 \mfb^3$ for some integral ideals $\mfa, \mfb$, we also have
\begin{align*}
\sum_{\substack{\mfr \text{ sqfull}}} N\mfr^{-1/2 - \delta} &\le \zeta_F\big(2 (1/2 + \delta)\big) \cdot \zeta_F\big(3 (1/2 + \delta)\big)
\end{align*}
Applying \eqref{eq:zetaF_bnd} and the bound
\[\zeta_F(3(1/2 + \delta)) \le \zeta(3/2)^n \le e^n\]
shows this is no greater than
\[\tfrac{1}{2}\delta^{-1} \cdot e^{2n} \cdot 2 \cdot 3^{1/4} \cdot \left(\log e^3 \Delta_F\right)^n.\]
Taking $\delta = (\log H)^{-1}$ then suffices to prove the first result. The proof of the second inequality is analogous but simpler.
\end{proof}

\begin{theorem}
\label{thm:Holder}
Choose a positive integer $t$. Fix $F$, $Q$, $M$, $d$, and $f_1, \dots, f_M$ as in Theorem \ref{thm:bilinear_basic}.  We define $r$ and $n$ as in that theorem. Choose $H \ge 100$.

Take $f$ to be the totally multiplicative function defined over $F$ so $f(\mfp) = d$ for each prime $\mfp$ of $F$.  For each integral ideal $\mfa$, choose a complex coefficient $b_{\mfa}$. We assume that $b_{\mfa} = 0$ if $\mfa$ has rational norm greater than $H$ or is not squarefree. 

Choose a positive integer $t$ such that
\[H^t \ge Q^de^{16nd^4}\]
 Given coprime integral ideals $\mfb, \mfr$ of $F$ with $\mfb$ squarefree and $\mfr$ squarefull, take
\[G(\mfr, \mfb) = \sum_{\mfa_1 \dots \mfa_t = \mfb \mfr} f(\mfr) \cdot \left|b_{\mfa_1} \dots b_{\mfa_t}\right|\]
and take
\[A_{0t} = \left(H^{-t} \cdot \sum_{\mfr, \mfb} G(\mfr, \mfb)^2\right)^{1/2t}.  \]
Then
\begin{align*}
&\sum_{i = 1}^M \left| \sum_{\mfa} b_{\mfa}f_i(\mfa)\right|\\
&\le A_{0t} \cdot (\log H^t)^{\max(r-1, 1)/2t} \cdot  (2d \log Q)^{26d^2n/t} \cdot HM\cdot \left((\#E/M^2)^{1/{2t}} + H^{-1/4}A^{1/2t}Q^{d/4t}\right)
\end{align*}
with
\[A = \exp\left(4n^{1/2}d^2\sqrt{\log Q^{-d}H^t} + 8nd^2 \sqrt{\log Q}\right).\]
\end{theorem}
\begin{proof}
By H\"{o}lder's inequality, we have
\begin{equation}
\label{eq:initial_Holder}
\sum_{i = 1}^M \left| \sum_{N\mfa \le H} b_{\mfa}f_i(\mfa)\right| \le M^{\frac{t-1}{t}} \cdot \left(\sum_{i = 1}^M \left| \sum_{N\mfa \le H} b_{\mfa}f_i(\mfa)\right|^t\right)^{1/t}.
\end{equation}
There are complex numbers $c_1, \dots, c_M$ of magnitude  $1$ 
such that
\begin{equation}
\label{eq:Holder2}
\sum_{i = 1}^M \left| \sum_{N\mfa \le H} b_{\mfa}f_i(\mfa)\right|^t = \sum_{i = 1}^M \sum_{\mfa_1, \dots, \mfa_t} c_i \cdot b_{\mfa_1}\cdot \dots \cdot b_{\mfa_t} \cdot f_i(\mfa_1) \cdot \dots \cdot f_i(\mfa_t). \end{equation}
Take $f_i^*$ to be the totally multiplicative function on integral ideals of $F$ that equals $f_i$ on squarefree ideals. Then the right hand side of \eqref{eq:Holder2} equals

\[\sum_{\mfr, \mfb} \left(f(\mfr)\sum_{\mfa_1 \dots \mfa_t = \mfb \mfr} b_{\mfa_1} \dots b_{\mfa_t} \right) \left(\sum_{i = 1}^M c_if_i(\mfb) f_i^*(\mfr)f^{-1}(\mfr)\right),\]
where the sum is over pairs of coprime ideals $(\mfb, \mfr)$ satisfying $N\mfb \cdot N\mfr \le H^t$ with $\mfb$ squarefree and $\mfr$ squarefull. By Cauchy--Schwarz, this has magnitude at most
\begin{equation}
\label{eq:CS_to_Holder}
\left(\sum_{\mfr, \mfb} G(\mfr, \mfb)^2 \right)^{1/2} \cdot \left(\sum_{\mfr, \mfb}\left|\sum_{i = 1}^M c_if_i(\mfb) f_i^*(\mfr)f^{-1}(\mfr)\right|^2\right)^{1/2}.
\end{equation}
Take
\begin{align*}
&A_1 = (\# E) \cdot H^t(\log H^t)^{r-1} \cdot (2d \log Q)^{50d^2n}, \\
&A_2 = M^2H^{t/2}Q^{d/2} (2d \log Q)^{26d^2n}\cdot \exp\left(4n^{1/2} d^2 \sqrt{\log Q^{-d}H^t} + 8nd^2 \sqrt{\log Q}\right),\quad\text{and}\\
&A_3 = M^2H^{t/2}Q^{d/2} e^{8nd^4} \cdot\left(\log e^{20nd^4} Q^d\right)^{2nd^2}.
\end{align*}

By Theorem \ref{thm:bilinear_basic}, we have
\[\sum_{\substack{\mfb \\ N\mfb \cdot N\mfr \le H^t}} \left|\sum_{i = 1}^M c_if_i(\mfb) f_i^*(\mfr)f^{-1}(\mfr)\right|^2 \le A_1 /N \mfr + A_2 /N\mfr^{1/2}\]
unless $H^t/N\mfr$ is smaller than $Q^de^{16nd^4}$. In this case, we may instead apply Remark \ref{rmk:min_H} to
bound this sum by $A_3/N\mfr^{1/2}$. So Lemma \ref{lem:squarefull} gives
\begin{equation}
\label{eq:inner_Holder}
\sum_{\mfr, \mfb}\left|\sum_{i = 1}^M c_if_i(\mfb) f_i^*(\mfr)f^{-1}(\mfr)\right|^2 \le 2^n \cdot A_1  + \max(A_2, A_3) \cdot (\log Q)^{8n} \cdot \log H^t.
\end{equation}
We note that
\begin{align*}
&e^{8nd^4} \le \exp\left(2n^{1/2}d^2 \sqrt{\log Q^{-d} H^t}\right), \quad\text{and} \\
&\log e^{20nd^4}Q^d \le (2d \log Q)^4.
\end{align*}
So $A_3 \le A_2$, and the left hand side of \eqref{eq:inner_Holder} is at most
\begin{align*}
&(\#E) \cdot H^t \cdot (\log H^t)^{r-1} \cdot (2d \log Q)^{51d^2n}\\
&+ M^2 Q^{d/2}H^{t/2} \cdot \log H^t \cdot (2d\log Q)^{34d^2n} \cdot \exp\left(4n^{1/2}d^2 \sqrt{\log Q^{-d}H^t} + 8nd^2 \sqrt{\log Q}\right).
\end{align*}
This gives us an upper bound for \eqref{eq:CS_to_Holder}, which in turn allows us to bound the left hand side of \eqref{eq:initial_Holder}.
\end{proof}
Assuming the values $A_{0t}$ do not grow too quickly with $t$, the following corollary gives a nearly optimal choice of $t$.
\begin{corollary}
\label{cor:choose_t}
Take $f_1, \dots, f_M$, $Q$, $d$, $n$, and $H$ as in Theorem \ref{thm:Holder}. Take
\[t = \left \lceil \frac{\log (Q^d M^2)  + 100nd^4 \sqrt{\log QM}}{\log H}\right\rceil \quad\text{and}\quad M_0 = M^2/\left(\#E \cdot (\log HMQ^{2d})^{60nd^2}\right).\]
Suppose that $M_0 > 1$. Then
\begin{align*}
&\sum_{i = 1}^M \left| \sum_{\mfa} b_{\mfa}f_i(\mfa)\right| \\
&\quad\le 2A_{0t}HM \cdot \exp\left( \frac{ - \log M_0 \cdot \log H}{2d\log Q + 4 \log M + 200nd^4 \sqrt{\log QM} + 2\log H } \right),
\end{align*}
where $A_{0t}$ is defined as in Theorem \ref{thm:Holder}.
\end{corollary}

\begin{proof}
Choose a real number $a$ such that
\[H^t = Q^d M^2 \exp\left(and^4 \sqrt{\log QM}\right),\]
so $a \ge 100$. We have the inequality
\begin{align*}
\exp\left(4n^{1/2}d^2 \sqrt{\log Q^{-d} H^t}\right) &=\exp\left(4n^{1/2}d^2 \sqrt{and^4 \sqrt{\log QM} + \log M^2}\right) \\
&\le \exp\left( 4nd^4 \sqrt{a+2}\cdot \sqrt{\log QM}\right),
\end{align*}
so
\begin{align*}
& (Q^{-d}H^t)^{-1/2} \cdot \exp\left(4n^{1/2}d^2 \sqrt{\log Q^{-d}H^t} + 8nd^2 \sqrt{\log Q}\right) \\
& \le M^{-1}\cdot \exp\left(\left(-\tfrac{1}{2}and^4 + 4\sqrt{a+2} \cdot nd^4 + 8nd^4\right) \cdot \sqrt{\log QM}\right) \le M^{-1}.
\end{align*}
The sum over $E$ in Theorem \ref{thm:Holder} is thus no smaller than the sum off $E$, so we have
\begin{equation}
\label{eq:Holder_t_chosen}
\sum_{i = 1}^M \left| \sum_{\mfa} b_{\mfa}f_i(\mfa)\right| \le\, 2A_{0t} \cdot HM \cdot (M^2/\#E)^{-1/2t} \cdot (\log H^t)^{d^2/2t} \cdot (2d \log Q)^{26d^2n/t}.
\end{equation}
We also have the inequalities
\[\log H^t \le \log H +  \log \left(Q^dM^2\right) + 100nd^4 \sqrt{\log QM} \le \tfrac{1}{2}\log H \cdot ( \log Q^{2d}M)^6\]
and
\[t \le \frac{\log H + \log\left( Q^d M^2\right) + 100nd^4 \sqrt{\log QM}}{\log H}.\]
We now may apply \eqref{eq:Holder_t_chosen} to prove the desired inequality.
\end{proof}
\begin{rmk}
\label{rmk:prime_support}
In this corollary, suppose the coefficient $b_{\mfa}$ is nonzero only when $\mfa$ is prime, and that its magnitude at primes is bounded by $1$. Choose primes $\mfp_1, \dots, \mfp_t$ of $F$ of norm at most $H$, and write $\mfp_1 \dots \mfp_t$ as $\mfb \mfr$, where $\mfb$ and $\mfr$ are coprime, $\mfb$ is squarefree, and $\mfr$ is squarefull. . Then
\[G(\mfr, \mfb) \le \frac{t!}{\#\text{Aut}((\mfp_1, \dots, \mfp_t))} \cdot f(\mfr) \le t! \cdot d^{t},\]
where $\#\text{Aut}((\mfp_1, \dots, \mfp_t))$ is the number of permutations in $S_t$ that fix $(\mfp_1, \dots, \mfp_t)$.

We have
\[\sum_{\substack{\text{Multiset }\{\mfp_1, \dots, \mfp_t\}\\ N\mfp_i \le H \text{ for } i \le t}} \frac{1}{\#\text{Aut}((\mfp_1, \dots, \mfp_t))} \le \frac{1}{t!} \left(\pi_F(H)\right)^t,\]
where $\pi_F$ is the prime counting function for $F$. We thus find
\begin{equation}
\label{eq:A0t_bnd}
A_{0t} \le \left(t! \cdot d^{2t} \cdot \frac{\pi_F(H)^t}{H^t}\right)^{1/2t} \le d t^{1/2} \cdot \frac{\pi_F(H)^{1/2}}{H^{1/2}} \le \frac{3dt^{1/2}n^{1/2}}{(\log H)^{1/2}},
\end{equation}
where the last inequality follows from \cite[Theorem 26A]{Rosser41} and the bound $H \ge 100$.
\end{rmk}

\section{The averaged Chebotarev density theorem}
\label{sec:averaged}
Given a finite group $G$, we call a character $\chi: G \to \C$ a \emph{faithful monomial character} if it is induced from a $1$-dimensional representation of a subgroup of $G$ and is a sum of irreducible faithful characters of $G$. With this codified, the following lemma is an application of Theorem \ref{thm:faithful}.

\begin{lemma}
\label{lem:to_positive_monomial}
Choose a nontrivial Galois extension $K/F$ of number fields, take $d = [K:F]$, and take $\Delta$ to be the magnitude of the absolute discriminant of $K$. Take $\chi$ to be a faithful irreducible character of $\Gal(K/F)$. Then there is an integer $m \le d$, a sequence  $\phi_1, \dots, \phi_m$ of faithful monomial characters of $\Gal(K/F)$, and rational numbers $a_1, \dots, a_m$ of magnitude at most $d^{\frac{3}{2}(d-1)}$ such that
\begin{equation}
\label{eq:irred_to_monomial}
\chi = a_1\phi_1 + \dots + a_m \phi_m
\end{equation}
Furthermore, the degree of each $\phi_i$ is at most $d/2$, and $Q_{\phi_i}$ is at most $\Delta$.
\end{lemma}
\begin{proof}
Take $\chi_1, \dots, \chi_m$ to be the distinct faithful irredcuible characters of $\Gal(K/F)$. From Theorem \ref{thm:faithful}, there are faithful monomial characters $\phi_1, \dots, \phi_m$ spanning the $\QQ$ vector space generated by $\chi_1, \dots, \chi_m$. As a result, if we define an $m \times m$ integer matrix $A = (a_{ij})$ so that
\[\phi_i = \sum_{j \le m} a_{ij} \chi_j \quad\text{for } i \le m,\]
we find that $A$ is invertible, and so has nonzero integer determinant. Since each $\phi_i$ corresponds to a subrepresentation of the regular representation of $\Gal(K/F)$, we find that each $a_{ij}$ is at most $\sqrt{d}$. From the theory of adjugate matrices, we see that the entries in the inverse matrix $A^{-1}$ have magnitude bounded by 
\[d^{\frac{1}{2}m-1} \cdot (m-1)! \le d^{\frac{3}{2}(d-1)}.\]
This gives the bound on the coefficients in \eqref{eq:irred_to_monomial}.

Since $\Gal(K/F)$ is nontrivial and each $\phi_i$ is faithful monomial, we see that each $\phi_i$ must be induced from a subgroup of order at least $2$, and hence has degree at most $d/2$. Finally, the bound on $Q_{\phi_i}$ follows from the fact that $\phi_i$ corresponds to a subrepresentation of the regular representation of $\Gal(K/F)$.
\end{proof}

\begin{theorem}
\label{thm:actual_main}
Choose positive numbers $Q, H \ge 100$ and an integer $M \ge 100$, choose a number field $F$, and choose a positive integer $d$. Finally, for each prime $\mfp$ of $F$, choose a complex number $b_{\mfp}$ of magnitude at most $1$.

Then there is a list $K_1, \dots, K_{M-1}$ of number fields so that, if $K/F$ is a Galois extension of relative degree $d$ such that $K$ has discriminant at most $Q$, and if $K$ is not in the list $K_1, \dots, K_{M-1}$, we have
\[\left| \sum_{N\mfp \le H} b_{\mfp} \cdot \chi( \mfp) \right| \le \frac{c H }{\log H}\]
with
\[ c = 10nd^{\frac{3}{2}(d + 2)} \sqrt{\log QMH} \cdot \exp\left( \frac{(- \log M + 15nd^2 \log \log Q^{d}MH ) \cdot \log H}{d\log Q + 4 \log M  + 2\log H + 13nd^4 \sqrt{\log QM} } \right)\]
for any irreducible faithful character $\chi$ of $\Gal(K/F)$.

With the same number of exceptions, we have
\begin{equation}
\label{eq:all_H}
\left| \sum_{N\mfp \le H} b_{\mfp} \cdot \chi( \mfp) \right| \le \frac{c(H) H}{\log H}
\end{equation}
for any $H \ge 100$, where $c(H)$ is defined as
\[11nd^{\frac{3}{2}(d + 2)} \sqrt{\log QMH} \cdot \exp\left( \frac{- \log \left(M (\log Q^dMH)^{-27nd^2}\right) \cdot \log H \cdot \left(1 - \frac{\log H}{\log Q^dM^4 H^3}\right)}{d\log Q + 4 \log M  + 2\log H + 13nd^4 \sqrt{\log QM} }\right). \]
\end{theorem}
\begin{proof}
Take $\mathcal{K}$ to be the set of Galois degree $d$ extensions of $F$ whose absolute discriminant is at most $Q$. For each $K$ in $\mathcal{K}$, choose an irreduciblle faithful character $\chi_K$ for which
\begin{equation}
\label{eq:bad_rating}
\left|\sum_{N\mfp \le H} b_{\mfp} \cdot \chi_K(\mfp) \right|
\end{equation}
is maximized. 

If $\mathcal{K}$ does not contain $M$ entries, the result is vacuous. Otherwise, we choose the $M$ fields $K_1, \dots, K_M$ in $\mathcal{K}$ for which the sum \eqref{eq:bad_rating} is maximized. To prove the first part, it  suffices to show
\[\sum_{i \le M}\left| \sum_{N\mfp \le H} b_{\mfp} \cdot \chi_{K_i}( \mfp) \right| \le \frac{c MH }{\log H}.\]
By Lemma \ref{lem:to_positive_monomial}, this will follow if we have
\[\sum_{i \le M}\left| \sum_{N\mfp \le H} b_{\mfp} \cdot \phi_i( \mfp) \right| \le \frac{d^{-\frac{3}{2}(d - 1) - 1}c MH }{\log H},\]
for any sequence of faithful monomial characters $\phi_1, \dots, \phi_M$, where $\phi_i$ is defined on $\Gal(K_i/F)$.

Recall from Lemma \ref{lem:to_positive_monomial} that the characters $\phi_i$ have degree at most $d/2$.
If we choose $t$ as in Corollary \ref{cor:choose_t}, \eqref{eq:A0t_bnd} gives
\[A_{0t} \le \frac{5n d^3 \sqrt{\log QMH}}{\log H}.\]
For $i \ne j$, we may view $\phi_i$ and $\phi_j$ as linear combinations of irreducible characters on $\Gal(K_iK_j/F)$ with kernel $\Gal(K_iK_j/K_i)$ and $\Gal(K_iK_j/K_j)$, respectively. From this, we have 
\[\langle \phi_i, \phi_j \rangle_{\Gal(K_iK_j/F)} = 0.\]
Corollary \ref{cor:choose_t} then gives the first part of the theorem.

For the second part, we first note that it suffices to show that, for any $H_0 \ge 100$, we have \eqref{eq:all_H} for all $H$ in $[H_0, 2H_0]$ and all fields in $\mathcal{K}$ with at most $M/ (\log H_0)^2$ exceptions, as we have
\[\sum_{k \ge 0} (\log(100 \cdot 2^k))^{-2} < 1.\]
Given $H_0$, we will choose an integer $M_1 \ge 1$ and apply the first part of the theorem with $(H, M)$ set to
\[\left(H_0(1 + k M_1^{-1}), \left\lceil \frac{M}{M_1(\log H_0)^2}\right\rceil\right)\]
for each integer $k$ in $[0, M_1)$. Noting that
\[\pi_F(H_0(1 + (k+1) M_1^{-1})) - \pi_F(H_0(1 + k M_1^{-1})) \le nM_1^{-1}H_0,\]
we find that we may take $c(H)$ equal to
\begin{align*}
&ndM_1^{-1}H \log H \\
&+ 10nd^{\frac{3}{2}(d + 2)} \sqrt{\log QMH} \cdot \exp\left( \frac{(- \log M/M_1 + 17nd^2 \log \log Q^{d}MH ) \cdot \log H}{d\log Q + 4 \log M  + 2\log H + 13nd^4 \sqrt{\log QM} } \right).
\end{align*}
The choice 
\[M_1 = \left\lfloor\log H \cdot \exp\left(\frac{- \log M \cdot \log H}{\log Q^dM^4H^3}\right) \right\rfloor\]
then gives the part.
\end{proof}

\subsection{The transition to the unconditional Chebotarev density theorem}
At some point as $H$ increases, the error term of Theorem \ref{thm:actual_main} becomes worse than what may be proved from the unconditional Chebotarev density theorem. Because of this, we need some form of the unconditional Chebotarev density theorem to prove Theorem \ref{thm:sparse_bad}.

This starts with a consideration of exceptional real zeros of $L$-functions.
\begin{defn}
\label{defn:exc_field}
Given a nontrivial quadratic extension $K/F$, we say $K$ is \emph{exceptional} if the Hecke $L$-function $L(s, K/F)$ corresponding to $K/F$ has a real root $\beta$ satisfying
\[1 - (32 \log |\Delta_K|)^{-1} \le \beta < 1.\]
If this root exists, it is necessarily simple \cite[Lemma 3]{Stark74}. We take $\mathbb{X}_{\textup{exc}}(F)$ to be the set of exceptional fields, and for an exceptional field $K$ we take $\beta(K)$ to be the real root defined as above.
\end{defn}
\begin{lemma}
\label{lem:few_exc_fields}
Given $\Delta \ge 3$, there is at most one field $K \in \mathbb{X}_{\textup{exc}}(F)$ such that
\[\Delta \le |\Delta_K| \le \Delta^2.\]
\end{lemma}
\begin{proof}
Suppose otherwise, so there were two distinct fields $K_1, K_2$ satisfying these conditions. 

Take $K_1K_2$ to be the composite field of $K_1$ and $K_2$. The function
\[\frac{\zeta_{K_1K_2}(s)}{\zeta(s)\cdot L(s, K_1/F) \cdot L(s, K_2, F)}\]
is a Hecke $L$-function and is hence entire. So $\zeta_{K_1K_2}$ has at least two roots counted with multiplicity in the interval $[1 - (32 \log |\Delta_K|)^{-1}, \,1)$.

But $K_1K_2$  has discriminant at most $\Delta^8$, so $\zeta_{K_1K_2}$ has at most one necessarily simple root in the interval $[1 - 4 \log \Delta^8,\, 1]$  by \cite[Lemma 3]{Stark74}. This contradicts our assumption, giving the proposition.
\end{proof}

\begin{lemma}
\label{lem:no_quad_no_siegel}
Take $K/F$ to be a nontrivial Galois extension, and take $\chi$ to be a faithful monomial character on $\Gal(K/F)$. If $L(s, \chi)$ has a real zero in the interval $[1 - (32 \log |\Delta_K|)^{-1}, 1)$, then $K$ lies in $\mathbb{X}_{\textup{exc}}(F)$.
\end{lemma}
\begin{proof}
Suppose $L(s, \chi)$ has a zero $\beta$ in this interval. Take $E/L$ to be fields such that $\chi$ is induced from a linear surjective character $\psi$ on $\Gal(E/L)$. Then $\beta$ is a simple zero of $\zeta_E$, and cannot be a zero of $\zeta_L$. Since $\beta$ is also a zero of $L(s, \overline{\chi})$, we find that $E/L$ is quadratic.

By \cite[Theorem 3]{Stark74}, $E$ then contains a field $M$ in $\mathbb{X}_{\textup{exc}}(F)$ with $\beta(M) = \beta$ such that $M$ is not contained in $L$. It follows that $E = ML$, implying that $\psi$ is the restriction of the nontrivial quadratic character on $\Gal(M/F)$. So $\chi$ is monomial faithful only if $K = M$.
\end{proof}

\begin{theorem}
\label{thm:cheb_explicit}
There is an absolute $C_0 > 0$ so we have the following:

Take $K/F$ to be a Galois extension of number fields of degree $d$, and take $n$ to be the degree of $F$. Then, for any faithful monomial character $\chi$ of $\Gal(K/F)$, we have
\[ \left|\sum_{N \mfp \le H} \chi(\mfp) \right| \le C_0H^{\beta(K)} + C_0ndH \exp\left(-\frac{\log H}{104 \log e^{nd} \Delta_K}\right)  + C_0H \exp\left(-\sqrt{\frac{\log H}{832nd}}\right)\]
for all $H \ge 1$. Here, we take $\beta(K) = 1/2$ if $K$ is not in $\mathbb{X}_{\textup{exc}}(F)$.

\end{theorem}
\begin{proof}
Define a function $\omega: \mathbb{R}^{\ge 0} \to \mathbb{R}^{\ge 0}$ by
\[\omega(t) = \frac{1}{13 \log e^{nd}\Delta_K + 13 nd \log \max(1, t)}.\]
By \cite[Theorems 1 and 2]{Lee21}, if $L$ is a subfield of $K$ and $K$ has sufficiently large discriminant, the Dedekind zeta function $\zeta_L$ has at most one zero $\sigma + it$ with $t \ge 0$ and $\sigma > 1-\omega(t)$, with this zero necessarily real and simple.  From the Artin formalism, the same statement then also holds for $L(s, \chi)$. Take $\beta$ to be this zero if its exists, and take $\beta = 3/4$ otherwise.

Following \cite[(7.5)]{LOTZ}, we define $\eta(H)$ for $H \ge 1$ by
\[\eta(H) = \text{inf}_{t \ge 0} \left(\omega(t) \log H + \log \max(1, t)\right)\]
for $x \ge 1$.
A calculus exercise shows that this satisfies
\begin{equation}
\label{eq:calculus}
\eta(H) \ge \min\left(\frac{\log H}{13 \log e^{nd} \Delta_K},  \sqrt{\frac{\log H}{13nd}} \right).
\end{equation}
We now apply the proof of \cite[Lemma 7.3]{LOTZ}. The only adjustment needed is to account for the possible exceptional real zero, whose impact we bound using \cite[Lemma 2.2 (iv)]{Thorner19}. The lemma then gives
\[\left|\sum_{N \mfp \le H} \chi(\mfp) \right| \ll H^{\beta} + \frac{H}{\log H}e^{-\frac{1}{8} \eta(H)} \cdot \log e \Delta_K\,\,\text{ for }\,\, H \ge \max(5000, (\log \Delta_K)^{4}),\]
where the implicit constant is absolute.
Applying Lemma \ref{lem:no_quad_no_siegel} then gives
\[ \left|\sum_{N \mfp \le H} \chi(\mfp) \right| \ll H^{\beta(K)} + ndH \exp\left(-\frac{\log H}{104 \log e^{nd} \Delta_K}\right)  + H \exp\left(-\sqrt{\frac{\log H}{832nd}}\right)\]
for $H > (\log e^{nd}\Delta_K)^{104}$. For smaller $H$, we find that the result is trivially true.

For $K$ of small discriminant, Lee's result does not handle the zeros with imaginary part lying in $[-1, 1]$. But only finitely many fields have such a small discriminant, and each has only finitely many zeros with imaginary part in this range. The impact of such zeros is then accounted for by the final term of the above inequality.
\end{proof}

\subsection{The proof of Theorem \ref{thm:sparse_bad}}
\label{ssec:proof_sparse_bad}
Take $n$ to be the degree of $F$. Note that an $\epsilon$-bad extension is $\epsilon'$-bad for $\epsilon' > \epsilon$. So if we can prove the theorem with $C(F,d) =  150nd^2$ and
\begin{equation}
\label{eq:epsilon_okay}
\epsilon \log \Delta \ge 200 nd^2 \log \log \Delta,
\end{equation}
it will hold without this assumption for $\Delta \gg_{F, d} 1$ and $C(F, d) = 400nd^2$.

We will replace $\Delta \gg_{F, d} 1$ with the explicit assumptions 
\begin{alignat}{2}
\label{eq:Delta_large}
&\log \log \Delta &&\ge 10^6 \cdot (d \log d + \log n),\quad\text{and}\\
 \label{eq:C0_too}
 & \log\log \Delta &&\ge 6 \log C_0,
\end{alignat}
where $C_0 > 0$ is chosen to satisfy the conditions of Theorem \ref{thm:cheb_explicit}. 

Take $\mathcal{K}$ to be the set of degree $d$ $\epsilon$-bad Galois extensions of $F$ whose absolute discriminant is bounded by $\Delta$, and take $\mathcal{K}_0$ to be the set of such extensions outside $\mathbb{X}_{\text{exc}}(F)$ By Lemma \ref{lem:few_exc_fields}, we have
\[\left|\mathcal{K} \backslash \mathcal{K}_0\right| \le 1+ 2\log \log \Delta.\]
So it suffices to bound the size of $\mathcal{K}_0$.

Take

\[H_{\text{max}} = \exp\left(\frac{104}{3}\epsilon \left( \log e^{nd}\Delta\right)^2\right),\]
and suppose that $H \ge H_{\text{max}}$. Then we have
\begin{align*}
&\exp\left(\frac{\sqrt{\log H}}{29\sqrt{nd}}\right) \le \exp\left( \sqrt{\frac{\log H}{832nd}}\right) - \frac{\sqrt{\log H}}{6000 \sqrt{nd}} \quad\text{and}\\
&\exp\left(\frac{\sqrt{\epsilon \log H}}{18}\right) \le  \frac{\log H}{18 \sqrt{104/3} \cdot \log e^{nd} \Delta } \le \frac{\log H}{104\cdot \log e^{nd}\Delta}- \frac{\log H}{6000\cdot \log e^{nd}\Delta}.
\end{align*}
Meanwhile, Theorem \ref{thm:cheb_explicit} and Lemma \ref{lem:to_positive_monomial} give the estimate
\[ \left|\sum_{N \mfp \le H} \chi(\mfp) \right| \le 3C_0nd^{3/2(d+1)}H \max\left(\exp\left(-\frac{\log H}{104 \log e^{nd} \Delta}\right),  \exp\left(-\sqrt{\frac{\log H}{832nd}}\right)\right).\]
So \eqref{eq:bad_defn} cannot hold if
\[\min\left(\frac{\log H}{6000 \cdot \log e^{nd}\Delta},\, \frac{\sqrt{\log H}}{6000\sqrt{nd}} \right) \ge \log\left( C_0nd^{3d}\log H\right).\]
for $H \ge H_{\text{max}}$, and this holds by our explicit assumptions on $\Delta$.

We now will apply Theorem \ref{thm:actual_main} to bound the size of $\mathcal{K}_0$. Taking $H_{\text{min}} = (\log \Delta)^{2 + \frac{d}{2\epsilon}}$, for each $K$ in $\mathcal{K}_0$, we may find an irreducible faithful character $\chi$ of $\Gal(K/F)$ and an $H \in [H_{\text{min}}, H_{\text{max}}]$ such that \eqref{eq:bad_defn} holds. Taking
\[M = \Delta^{\epsilon(1 + \delta)} (\log \Delta)^{100nd^2}\quad\text{with}\quad \delta = \frac{100(d \log d + \log n)}{\log \log \Delta} + \frac{100 \sqrt{d}}{\sqrt{\log \log \Delta}},\]
we claim that $|\mathcal{K}_0| < M$. This is stronger than the claim of the theorem with $C(F, d) = 150nd^2$.

By Theorem \ref{thm:actual_main}, we will have proved the claim if we can show that
\[nd^{5d}\sqrt{\log \Delta MH} \exp\left(\frac{\sqrt{\epsilon \log H}}{18} - \frac{\log \left(M (\log \Delta^dMH)^{-27nd^2}\right) \cdot \log H \cdot \left(1 - \frac{\log H}{\log \Delta^d H^3}\right)}{d\log \Delta + 4 \log M  + 2\log H + 13nd^4 \sqrt{\log \Delta M} }\right)\]
is at most $1$ for all $H$ in the interval $I = [H_{\text{min}}, H_{\text{max}}]$. Calling this expression $f_0(H)$, we will prove this in the following four steps:
\begin{enumerate}
    \item Taking $H_0 = H_{\text{min}}^4$, we will show that $f_0(H) \le 1$ for $H$ in $[H_{\text{min}}, H_0]$.
    \item We will find a second function $f_1$ so  that $f_1(H) \ge f(H)$ for all $H$ in $[H_{\text{min}}, H_{\text{max}}]$.
    \item We will show that the minimal value attained by $f_1$ on $[H_0, H_{\text{max}}]$ is attained at either $H_0$ or $H_{\text{max}}$.
    \item We will check that $f_1(H_0) \le 1$ and that $f_1(H_{\text{max}}) \le 1$.
\end{enumerate}

Before proceeding with the first step, we will list a few estimates we will use multiple times.  First, by \eqref{eq:Delta_large}, we have $\delta \le 1/6$. From \eqref{eq:epsilon_okay}, we also know that $(\log \Delta)^{100nd^2}$ is at most $\Delta^{\epsilon/2}$. So we find $M \le \Delta^2$, and \eqref{eq:Delta_large} gives
\begin{equation}
\label{eq:13_estimate}
13nd^4 \sqrt{\log \Delta M} \le \frac{\delta}{100} \log \Delta.
\end{equation}
We also have $\log H_{\text{max}} \le 200 (\log \Delta)^2$, and so we find
\[\log (\Delta^dMH_{\text{max}}) \le 400 (\log \Delta)^2 \le (\log \Delta)^{\frac{20}{9}}\]
by \eqref{eq:Delta_large}. So
\begin{equation}
\label{eq:M_est}
\left(M (\log \Delta^dMH)^{-27nd^2}\right) \ge \Delta^{\epsilon(1 + \delta)}( \log \Delta)^{40nd^2}
\end{equation}
for $H$ on the interval $I$.

With these set, we will prove the first itemized claim. Note that
\begin{equation}
    \label{eq:H0_bnd}
    H_0 \le  (\log \Delta)^{\frac{6d}{\epsilon}},
\end{equation} 
so $H_0 \le \Delta^{1/30}$ by \eqref{eq:epsilon_okay} and
\[nd^{5d} \sqrt{\log \Delta M H_0} \le 2nd^{5d}  \sqrt{\log \Delta} \le 2 \sqrt{\log \Delta} \cdot \exp\left(\frac{5d \log d + \log n}{\log \log \Delta} \cdot \log \log \Delta \right).\]
 From \eqref{eq:H0_bnd}, we also have
 \begin{equation}
 \label{eq:h0_sqrt}
\frac{\sqrt{\epsilon \log H_0}}{18} \le  \frac{ \sqrt{d}}{\sqrt{\log \log \Delta}} \log \log \Delta.
\end{equation}
These last inequalities imply
\[nd^{5d}\sqrt{\log \Delta MH} \exp\left(\frac{\sqrt{\epsilon \log H}}{18}\right) \le \exp\left(\left(\frac{1}{2} + \frac{\delta}{20} \right) \log \log \Delta\right). \]
The first claim will then be shown if we can prove
\[\log \frac{\log \left(\Delta^{\epsilon(1 + \delta)}( \log \Delta)^{40nd^2}\right) \cdot \log H_{\text{min}} \cdot \left(1 - \frac{\log H_0}{\log \Delta^d H_0^3}\right)}{d\log \Delta + 4 \log M  + 2\log H_0 + \frac{\delta}{100} \log \Delta } \ge \log \left(\frac{1}{2} + \frac{\delta}{20} \right) + \log \log \log \Delta.\]
We will repeatedly use
\begin{equation}
    \label{eq:log_estimates}
    \log(1 + x) \le x \,\,\text{ for } \,\,0 \le x \quad\text{and}\quad \log(1 + x) \ge \frac{4}{5}x \,\,\text{ for }\,\, 0 \le x \le 1/2.
\end{equation}

Since $\delta \le 1/6$, \eqref{eq:log_estimates} and \eqref{eq:epsilon_okay} give
\[\log \log \left(\Delta^{\epsilon(1 + \delta)} (\log \Delta)^{40nd^2}\right)  \ge \log (\epsilon \log \Delta) + \frac{4}{5}\delta + \frac{32nd^2 \log \log \Delta}{\epsilon \log \Delta}.\]
We also have
\begin{align*}
&\log \log H_{\text{min}} = \log(2 + d/2\epsilon) + \log \log \log \Delta, \\
&\log \left(1 - \frac{\log  H_0}{\log \Delta^dH_0^3}\right) \ge - \log \frac{\log \Delta^d H_0}{\log \Delta^d} \ge -\frac{\log H_0}{d \log \Delta} \ge -\frac{2nd^2 \log \log \Delta}{\epsilon \log \Delta}, \quad\text{ and}\\
& \log\left( d\log \Delta + 4 \log M  + 2\log H_0 + \frac{\delta}{100} \log \Delta \right) \\
& \qquad \le \log \left((d + 4\epsilon) \log \Delta\right) + \frac{4\delta\epsilon}{d + 4 \epsilon} + \frac{400nd^2 \log \log \Delta}{d\log \Delta} + \frac{12 \log \log \Delta}{\epsilon \log \Delta} + \frac{\delta}{100}. 
\end{align*}
We note that
\[\frac{400nd^2 \log \log \Delta}{d\log \Delta} \le \frac{\delta}{100}\quad\text{and}\quad \frac{4\delta\epsilon}{d + 4\epsilon} \le \frac{2}{3} \delta,\]
so summing these estimates for logarithms gives the first claim. We also note that the estimates for the denominator and the bound $\delta \le 1/6$ give
\begin{equation}
\label{eq:M_contrib}
4\log M + \frac{\delta}{100} \log \Delta \le 4\epsilon \log \Delta + (d + 4\epsilon) \cdot \frac{\log \Delta}{8}.
\end{equation}

We now consider the second step. We will take
\[f_1(H) = (\log \Delta)^{10/9} \cdot \exp\left(\frac{\sqrt{\epsilon \log H}}{18} - \frac{\tfrac{2}{3} \log \left(\Delta^{\epsilon }\right) \cdot \log H }{\tfrac{9}{8} (d + 4\epsilon) \log \Delta  + 2\log H }\right).\]

That this is greater than $f(H)$ on the interval $I$ follows from \eqref{eq:M_contrib} and the estimates
\begin{align*}
&nd^{5d} \sqrt{\log \Delta MH_{\text{max}}} \le (\log \Delta)^{10/9}\quad\text{and}\quad  1 - \frac{\log H_0}{\log \Delta^d H_0^3} \ge \frac{2}{3}.
\end{align*}

We now move to the third step. Define a function
\[f_2(x) = \frac{x \cdot \sqrt{\epsilon}}{18} - \frac{\tfrac{2}{3} \log \left(\Delta^{\epsilon}\right) \cdot x^2}{\tfrac{9}{8} (d + 4\epsilon) \log \Delta  + 2x^2 }.\]
We claim that the derivative $f_2'(\sqrt{\log H_0})$ is negative, that the second derivative $f_2''(x)$ is zero at for a single choice of $x > 0$, and that $f_2''$ is negative before this zero and positive after this zero. These claims taken together will imply that the maximum value attained by $f_2$ on $[\sqrt{\log H_0}, \sqrt{\log H_{\text{max}}}]$ is attained at one of its endpoints.

We have
\[f_2'(x) =  \frac{ \sqrt{\epsilon}}{18} - \frac{\tfrac{4}{3} \log \left(\Delta^{\epsilon}\right)\cdot \tfrac{9}{8} (d + 4\epsilon) \log \Delta \cdot x }{\left(\tfrac{9}{8} (d + 4\epsilon) \log \Delta  + 2x^2\right)^2}.\]
Since $H_0 \le \Delta^{1/30}$, this is at most
\[\frac{\sqrt{\epsilon}}{18} - \frac{ \epsilon \log \Delta  \cdot d\log \Delta \cdot \sqrt{\epsilon^{-1} d \log \log \Delta}}{(5d \log \Delta)^2} \]
at $\sqrt{\log H_0}$, and this is negative by \eqref{eq:Delta_large}.

The claims about the second derivative hold for any function of the shape $a_0x  - a_1 + \frac{a_2}{a_3 + x^2}$ for positive constants $a_0, a_1, a_2, a_3$. This finishes the third step.

So it only remains to show that $f_1$ is at most $1$ at $H_0$ and at $H_{\text{max}}$. At $H_0$, we use that $H_0 \le \Delta^{1/30}$ to write
\[\frac{\tfrac{2}{3} \log \left(\Delta^{\epsilon}\right) \cdot \log H_0 }{\tfrac{9}{8} (d + 4\epsilon) \log \Delta  + 2\log H_0 } \ge \frac{ \tfrac{2}{3} \epsilon \log \Delta  \cdot 4(2 + \tfrac{d}{2\epsilon}) \cdot \log\log \Delta}{\left(\tfrac{9}{8} + \tfrac{1}{30} \right) \cdot (d + 4\epsilon)\cdot \log \Delta} \ge 1.15 \cdot  \log \log \Delta. \]
Together with the bound \eqref{eq:h0_sqrt}, we find that $f_1(H_0) \le 1$.

We now consider $ \log f_1(H_{\text{max}})$. Since $1/(1 + x)$ is at least $1-x$ for $x$ positive, this is at most
\begin{align*}
&\frac{10}{9} \log \log \Delta + \frac{\sqrt{\frac{104}{3}} \epsilon \log e^{nd} \Delta}{18} - \frac{1}{3} \epsilon \log \Delta + \frac{3(d + 4 \epsilon) \log \Delta}{16 \cdot \log H_{\text{max}}} \\
& \le nd + 2 \log \log \Delta - \frac{1}{170} \epsilon \log \Delta + \frac{d}{60 \epsilon \log \Delta} \le 0,
\end{align*}
with the final inequality following from \eqref{eq:epsilon_okay}. So $f_1(H_{\text{max}}) \le 1$.
\qed

\subsection{The averaged Chebotarev density theorem}
We will prove the following strengthened form of Proposition \ref{prop:avg_cdt}.

\begin{theorem}
\label{thm:avg_cdt}
Given a number field $F$ of degree $n$, and given a positive integer $d$, there is some $C(n, d) > 0$ depending just on $n$ and $d$ so we have the following

Choose any $\epsilon > 0$ and a Galois extension $K/F$ of degree $d$. Choose $H > (\log \Delta)^{2 + \frac{d}{2\epsilon}}$. For every $\epsilon$-bad extension $L/F$ contained in $K$, we assume that
\[\log H \ge C(n, d) \cdot \left(\log 3\Delta_L\right)^2.\]
If this bad $L$ lies in $\mathbb{X}_{\textup{exc}}(F)$, we also assume that
\[1 - \beta(L) \ge \frac{1}{40 \sqrt{\log H}}.\]
Then, for any conjugacy class $C$ of $G = \Gal(K/F)$, we have

\[\left|\pi_C(H; K/F) \,-\, \frac{|C|}{|G|} \cdot \pi_F(H)\right| \le  \frac{H}{\log H} \cdot \exp\left(- c(\epsilon) \cdot \sqrt{\log H}\right).\]
\end{theorem}

The most difficult aspect of the proof of this theorem is that its right hand side exactly matches the right hand side of Definition \ref{defn:eps_bad}, necessitating the following lemma.
\begin{lemma}
\label{lem:fun}
Take $C$ to be a conjugacy class of a finite group $G$, and take $\chi: G \to \C$ to equal $1$ on $C$ and $0$ outside $C$. Take $\chi_1, \dots, \chi_m$ to be the irreducible characters of $G$. Then there are coefficients $a_1, \dots, a_m \in \C$ such that
\[\chi = \sum_i a_i \chi_i \quad\text{and}\quad \sum |a_i| \le 1.\]
\end{lemma}
\begin{proof}
Since the $\chi_i$ give a basis for the set of class functions of $G$, we have $\chi = \sum_i a_i \chi_i$ for a unique choice of $a_1, \dots, a_m$. We now need to show that $\sum |a_i| \le 1$.

Define the inner product $\langle\,,\,\, \rangle$ on class functions of $G$ as in Section \ref{ssec:first_reductions}. Take $b = |G|/|C|$. Then
\[b^{-1} = \langle\chi, \chi\rangle = |a_1|^2 + \dots + |a_m|^2.\]
Take $S$ to be the set of nonzero $a_i$. The numbers $ba_i$ are algebraic integers, and the multiset $\{ba_i\,:\,\, i \in S\}$ must be the set of roots of some monic integer polynomial $P$ satisfying $P(0) \ne 0$. So the AM-GM inequality gives
\[b = \sum_{i \in S} \left|ba_i\right|^2 \ge |S|\cdot  \prod_{i \in S}  |ba_i|^{2/|S|} = |S| \cdot |P(0)|^{2/|S|} \ge |S|. \]
The result now follows from the Cauchy--Schwarz inequality.
\end{proof}

\begin{proof}[Proof of Theorem \ref{thm:avg_cdt}]
By Lemma \ref{lem:fun}, it suffices to show that
\begin{equation}
\label{eq:chi_of_acdt}
\left|\sum_{N\mfp \le H} \chi(\mfp)\right| \le \frac{H}{\log H} \cdot \exp\left(- c(\epsilon)\cdot \sqrt{\log H}\right)
\end{equation}
for every irreducible character $\chi$ of $G$. For a given $\chi$, this claim is clear unless the subfield $L$ of $K$ fixed by the kernel of $\chi$ is $\epsilon$-bad.

Now suppose $L$ is $\epsilon$-bad. If $L$ lies in $\mathbb{X}_{\textup{exc}}(F)$, we have
\[H^{\beta(L)} \le H \cdot \exp\left(- \frac{\sqrt{\log H}}{40}\right)\]
by the assumptions on $\beta(L)$. Note that $c(\epsilon) < 1/41$, so
\[C_0 H^{\beta(L)} \le \frac{H}{2\log H} \exp\left( - c(\epsilon)\sqrt{\log H}\right)\]
so long as $H$ is larger than some absolute constant, where $C_0$ is defined as in Theorem \ref{thm:cheb_explicit}. The inequality \eqref{eq:chi_of_acdt} then follows from Theorem \ref{thm:cheb_explicit} for a proper choice of $C(n, d)$.
\end{proof}

We now turn to the proof of Proposition~\ref{prop:avg_pit}, which, like Theorem~\ref{thm:avg_cdt}, we prove in a slightly strengthened form.

\begin{theorem}\label{thm:avg_pit}
    Let $F$ be a number field of degree $n$ and let $m \geq 2$ be an integer.  Let $\epsilon>0$, let $L/F$ be an extension of degree $m$, let $K$ be the normal closure of $L/F$, and assume that $\Delta_K \leq \Delta$, where $\Delta$ is as above.  Let $G=\mathrm{Gal}(K/F)$.  Let $H > (\log \Delta)^{2 + \frac{|G|}{2\epsilon}}$ be such that for every $\epsilon$-bad extension $M/F$ not linearly disjoint from $L$ over $F$, we have
        \[
            \log H \geq C(n, |G|)\cdot (\log 3\Delta_M)^2,
        \]
    where $C(n, |G|)$ is as in Theorem~\ref{thm:avg_cdt}.  If this bad $M$ lies in $\mathbb{X}_\mathrm{exc}(F)$, we also assume that
        \[
            1 -\beta(L) \geq \frac{1}{40 \sqrt{\log H}}.
        \]
    Then we have
        \[
            \left| \pi_L(H) - \pi_F(H) \right|
                \leq \frac{H}{\log H} \cdot (m-1) \cdot \exp\left(-c(\epsilon) \sqrt{\log H}\right).
        \]    
\end{theorem}
\begin{proof}
    Let $H = \mathrm{Gal}(K/L)$ and let $\chi_H = \mathrm{Ind}_H^G 1 - 1$.  We then have that
        \[
            \pi_L(H) - \pi_F(H)
                = \sum_{\mathrm{N}\mfp \leq H} \chi_H(\mfp).
        \]    
    Since $\chi_H(1) = m-1$, it follows that $\chi_H$ admits at most $m-1$ irreducible constituents $\chi$.  The result therefore follows if we show for each of these constituents $\chi$ that
        \[
            \left|\sum_{N\mfp \le H} \chi(\mfp)\right| \le \frac{H}{\log H} \cdot \exp\left(- c(\epsilon)\cdot \sqrt{\log H}\right).
        \]
    As in the proof of Theorem~\ref{thm:avg_cdt}, this is straightforward unless the kernel field of $\chi$ is $\epsilon$-bad.  Let $M = K^{\ker \chi}$ be this kernel field, and observe that $M$ and $L$ are not linearly disjoint (for example, this follows from \cite[Lemma 3.9]{LOTZ}).  Hence, proceeding exactly as in the proof of Theorem~\ref{thm:avg_cdt}, the result follows.
\end{proof}

To obtain Proposition~\ref{prop:avg_pit} in the form stated in the introduction, we require the following lemma, after which the proof of Proposition~\ref{prop:avg_pit} is routine.

\begin{lemma} \label{lem:galois-disc}
    Let $F$ be a number field, let $L/F$ be a finite extension, and let $K/F$ be its normal closure.  Let $G = \mathrm{Gal}(K/F)$ and $m = [L:F]$.  Then $\Delta_K \leq \Delta_L^{|G|/2} \Delta_F^{-|G|(m-2)/2}$.
\end{lemma}
\begin{proof}
    This is \cite[Lemma 3.10]{LOTZ}.
\end{proof}
\section{Arithmetic applications}
\label{sec:arithmetic}
In this section, we provide the proofs of the arithmetic applications of the averaged Chebotarev density theorem.
The following result makes clear the role that primitivity will play in these applications.

\begin{lemma} \label{lem:primitive-not-bad}
    Let $F$ be a number field, let $L/F$ be a primitive extension, and let $K/F$ be its normal closure.  Suppose for some $\epsilon>0$ that $K$ is not $\epsilon$-bad.  Then $L$ is linearly disjoint from every extension in $\mathbb{X}_\mathrm{bad}(F,\epsilon)$ contained in $K$.  In particular, $L$ is subject to Proposition~\ref{prop:avg_pit}.
\end{lemma}
\begin{proof}
   This is immediate from the definition of a primitive extension. 
\end{proof}

\subsection{Bounds and moments for $\ell$-torsion subgroups: Proof of Corollary~\ref{cor:ellenberg-venkatesh} and Corollary~\ref{cor:koymans-thorner}}
We begin by recalling a lemma of Ellenberg and Venkatesh \cite[Lemma 2.3]{EllenbergVenkatesh} that will, together with Proposition~\ref{prop:avg_pit}, readily imply Corollary~\ref{cor:ellenberg-venkatesh}.

\begin{lemma}[Ellenberg--Venkatesh]
    \label{lem:ellenberg-venkatesh}
    Let $L/F$ be a degree $m$ extension of number fields, let $\ell$ be a positive integer, and let $\delta< \frac{1}{2\ell(m-1)}$.  Let $M$ be the number of prime ideals $\mathfrak{p}$ of $L$ with norm at most $\Delta_{L/F}^{\delta}$ that are not extensions of prime ideals from any proper subextension of $L/F$, where $\Delta_{L/F}$ denotes the norm of the relative discriminant of $L/F$.  Then for any $\varepsilon > 0$, there holds
        \[
            |\mathrm{Cl}(L)[\ell]|
                \ll_{F,m,\ell,\delta,\varepsilon} \Delta_L^{\frac{1}{2}+\varepsilon}/M.
         \]
\end{lemma}

Using this, we now prove Corollary~\ref{cor:ellenberg-venkatesh}.

\begin{proof}[Proof of Corollary~\ref{cor:ellenberg-venkatesh}]
    Let $\ell,m \geq 2$ be integers as in the statement of the corollary and let $F$ be a number field.  Let $Q_0$ be the least real number such that: $Q_0 \geq \Delta_F^{2m} \exp\left( 4096 \ell^2 n^4 (m-1)^2 \cdot (m!)^6 \right)$; we may take $C(F,m!) = 400 n (m!)^2$ in Theorem~\ref{thm:sparse_bad} for every $\Delta \geq Q_0^{m!}$; and
        \[
            \pi_F(H) - \pi_F(H^{1/m})
                \geq \frac{99}{100} \frac{H}{\log H}
        \]
    for every $H \geq Q_0^{\frac{1}{8\ell(m-1)}}$.  
    Since $Q_0 \ll_{F,m,\ell} 1$, the statement of the theorem is trivial if $Q \leq Q_0$.  Thus, we may assume $Q > Q_0$, and we may similarly restrict our attention to those extensions $L$ such that $\Delta_L > Q_0$.  Now, for each positive integer $j \leq \log Q - \log Q_0 + 1$, let $Q_j := e^{j-1} Q_0$,  $\epsilon_j := 16 \ell(m-1)\cdot m!  \cdot \frac{\log\log Q_j}{\log Q_j}$, and $\mathcal{E}_j$ be the subset of $\mathscr{F}_{m,F}^\mathrm{prim}(Q)$ consisting of those $L$ with $Q_j < \Delta_L \leq e\cdot Q_j$ that are not linearly disjoint from every field in $\mathbb{X}_\mathrm{bad}(F,\epsilon_j)$.

    Fix some $L \in \mathscr{F}_{m,F}^\mathrm{prim}(Q)$ with $\Delta_L > Q_0$.  Then there is a unique $j$ such that $Q_j < \Delta_L \leq e Q_j$, and we suppose that $L \not\in \mathcal{E}_j$, i.e. that $L$ is linearly disjoint from every field in $\mathbb{X}_\mathrm{bad}(F,\epsilon_j)$.  We aim to show in this case that Proposition~\ref{prop:avg_pit} applies meaningfully in the range required by Lemma~\ref{lem:ellenberg-venkatesh}.  To this end, we first observe that by our choice of $Q_0$, we have that $\frac{|G|}{2} \log \Delta_L \leq (\log Q_j)^2$, and hence that
        \[
            \left(\frac{|G|}{2} \log \Delta_L\right)^{2 + \frac{|G|}{2\epsilon_j}}
                \leq (\log Q_j)^{4 + \frac{m!}{\epsilon_j}}
                = Q_j^{\frac{1}{16\ell (m-1)} + \frac{4 \log\log Q_j}{\log Q_j}}
                < Q_j^{\frac{1}{8\ell(m-1)}},
        \]
    since $\frac{\log\log Q_j}{\log Q_j} \leq (\log Q_j)^{-1/2} \leq (\log Q_0)^{-1/2} \leq \left(64 \ell n^2 (m-1) \cdot (m!)^3\right)^{-1}$.  In particular, we may apply Proposition~\ref{prop:avg_pit} with any $H \geq Q_j^{\frac{1}{8\ell(m-1)}}$.  Since $c(\epsilon_j) = \frac{\sqrt{\epsilon_j}}{18}$, we find for any $H \geq Q_j^{\frac{1}{8\ell(m-1)}}$ that
        \[
            m \exp\left( - c(\epsilon_j) \sqrt{\log H} \right)
                < m \exp\left(-\frac{(m-1) \cdot \sqrt{2}}{18} \sqrt{\log\log Q_0} \right)  < \frac{49}{50},
        \]
    since our assumptions imply that $\log Q_0 \geq 2^{20}$.  In particular, we conclude for any $H \geq Q_j^{\frac{1}{8\ell(m-1)}}$ that
        \begin{equation} \label{eqn:ev-pit-upshot}
            \left| \pi_L(H) - \pi_F(H) \right|
                < \frac{49}{50} \frac{H}{\log H}.
        \end{equation}
    Now, since the extension $L/F$ is primitive, the only prime ideals of $L$ that are the extension of an ideal from a proper subfield are those that are inert in the extension $L/F$.  There are at most $\pi_F(H^{1/m})$ such prime ideals of norm at most $H$, and by our assumptions on $Q_0$, we find from \eqref{eqn:ev-pit-upshot} that
        \begin{equation} \label{eqn:ev-non-inert}
            \pi_L(H) - \pi_F(H^{1/m})
                > \frac{1}{100} \frac{H}{\log H}
        \end{equation}
    for any $H \geq Q_j^{\frac{1}{8\ell(m-1)}}$.  Finally, since we have assumed that $Q_0 \geq \Delta_F^{2m}$, we find that $\Delta_{L/F} \geq \Delta_L^{1/2} > Q_j^{1/2}$.  Thus, for any fixed $\delta$ such that $\frac{1}{4\ell(m-1)} \leq \delta < \frac{1}{2\ell(m-1)}$, we find from Lemma~\ref{lem:ellenberg-venkatesh} and \eqref{eqn:ev-non-inert} that
        \[
            |\mathrm{Cl}(L)[\ell]|
                \ll_{F,m,\ell,\delta,\varepsilon} \Delta_L^{\frac{1}{2} - \delta + \varepsilon}.
        \]
    Letting $\delta$ tend to $\frac{1}{2\ell(m-1)}$ from below, we conclude that the bound
        \[
            |\mathrm{Cl}(L)[\ell]|
                \ll_{F,m,\ell,\varepsilon} \Delta_L^{\frac{1}{2} - \frac{1}{2\ell(m-1)} + \varepsilon}
        \]
    must hold provided that $L \not\in \mathcal{E}_j$.

    We therefore aim to bound the sizes of the sets $\mathcal{E}_j$.  If $Q_j < \Delta_L \leq e Q_j$, then by Lemma~\ref{lem:galois-disc}, we find that $\Delta_K \leq \Delta_L^{m!/2} < Q_j^{m!}$, where $K$ denotes the normal closure of $L/F$.  Hence, appealing to Theorem~\ref{thm:sparse_bad} with $\Delta = Q_j^{m!}$ and $d=m!$, we see that the number of possible extensions $K$ in $\mathbb{X}_\mathrm{bad}(F,\epsilon_j)$ is at most
        \[
            (\log Q_j)^{16\ell (m-1) \cdot (m!)^2 + \frac{6400 n\ell (m-1) (m!)^4}{(\log\log Q_j)^{1/2}} + 800 n (m!)^2} < (\log Q)^{2188 n \ell (m-1) \cdot (m!)^4},
        \]
    where we have once again used that $\log Q_0 > 2^{20}$.  The number of extensions $L$ with the same normal closure $K$ is at most the number of subgroups of the symmetric group $S_m$, which we may bound trivially by $2^{m!} < (\log Q)^{m!}$.  Accounting for the $\log Q - \log Q_0 + 1 < \log Q$ different values $j$, we conclude in sum that
        \[
            \left| \bigcup_{j \leq \log Q - \log Q_0 + 1} \mathcal{E}_j\right|
                < (\log Q)^{2200 n \ell (m-1) \cdot (m!)^4},
        \]
    which yields the corollary with the explicit value $A = 2200 \cdot n \ell (m-1) \cdot (m!)^4$.
\end{proof}

Turning to Corollary~\ref{cor:koymans-thorner}, the following proposition summarizes the methods of Koymans and Thorner \cite{KoymansThorner}.

\begin{proposition}[Koymans--Thorner]
    \label{prop:koymans-thorner}
    Let $F$ be a number field and let $\mathcal{S}$ be any set of extensions $L/F$, all of which have the same degree $m$.  For any $Q \geq 1$, let $\mathcal{S}(Q) := \{ L \in \mathcal{S} : \Delta_L \leq Q\}$.  Suppose for any $\varepsilon>0$, there are constants $c_1,c_2 > 0$ (depending on $F$, $\mathcal{S}$, and $\varepsilon$) such that for any $Q \geq 1$, there is a subset $\mathcal{E}(Q) \subseteq \mathcal{S}(Q)$ satisfying $|\mathcal{E}(Q)| = O_{F,\mathcal{S},\varepsilon}(Q^\varepsilon)$ such that whenever $L \in \mathcal{S}(Q) \setminus \mathcal{E}(Q)$, we have for any $x \geq (\log Q)^{c_1}$ that
        \[
            \pi_L(x) \geq c_2 \frac{x}{\log x}.
        \]
    Then for any integers $\ell \geq 2$, $r\geq 1$, and any $Q \geq 1$ and $\varepsilon > 0$, there holds
        \[
            \sum_{L \in \mathcal{S}(Q)} |\mathrm{Cl}(L)[\ell]|^r
                \ll_{F,\mathcal{S},\ell,r,\varepsilon} Q^{\frac{r}{2}+\varepsilon}\left( 1 + |\mathcal{S}(Q)|^{1-\frac{r}{\ell(m-1)+1}}\right).
        \]
\end{proposition}
\begin{proof}
    No proposition of this form is stated explicitly in the work of Koymans and Thorner \cite{KoymansThorner}, but it is implicit and easily obtained from their work, as we now explain.  First, if we let $\pi_L^{(1)}(x)$ denote the number of degree $1$ prime ideals of $L$ with norm at most $x$, then we have
        \[
            \pi_L^{(1)}(x)
                \geq \pi_L(x) - m [F:\mathbb{Q}] \pi(x^{1/2})
                \geq \pi_L(x) - m[F:\mathbb{Q}] x^{1/2}.
        \]
    As a result, we obtain for $\varepsilon > 0$ and $Q$ sufficiently large that any $L \in \mathcal{S}(Q) \setminus \mathcal{E}(Q)$ satisfies
        \begin{equation} \label{eqn:KT-hypothesis}
            \pi_L^{(1)}(x) \geq \frac{c_2}{2} \frac{x}{\log x}
        \end{equation}
    provided that $x \geq (\log Q)^{c_1}$.  The claim then follows as in the proof of \cite[Theorem 1.1]{KoymansThorner}.  More specifically, the proof of their Theorem 1.1 relies on their Theorem 3.3, Lemma 4.1, and Corollary 5.2.  Of these, only Corollary 5.2 makes use of the specific families that Koymans and Thorner study.  The statement of their Corollary 5.2 is essentially equation \eqref{eqn:KT-hypothesis} but for the specific families of interest to them.  Thus, replacing Corollary 5.2 by \eqref{eqn:KT-hypothesis} in their proof, the result follows.
\end{proof}

Appealing to Proposition~\ref{prop:avg_pit} and Theorem~\ref{thm:sparse_bad}, we see that the hypothesis of Proposition~\ref{prop:koymans-thorner} is satisfied for the family $\mathcal{S}=\mathscr{F}_{m,F}^\mathrm{prim}$.  This immediately implies Corollary~\ref{cor:koymans-thorner}.  However, we note that it also implies that the hypothesis of Proposition~\ref{prop:koymans-thorner} is satisfied for finer sets of primitive extensions.  In particular, let $G$ be a primitive permutation group of degree $n$.  (Recall that a permutation group is called primitive if it preserves no nontrivial partition of the underlying set.)  Given any $L \in \mathscr{F}_{m,F}^\mathrm{prim}$, the Galois group $\mathrm{Gal}(\widetilde{L}/F)$ of its normal closure over $F$ acts on the $m$ embeddings of $L$ into $\widetilde{L}$, or, essentially equivalently, on the $m$ cosets of $\mathrm{Gal}(\widetilde{L}/L)$.  We let $\mathscr{F}_{m,F}^G$ be the subset of $\mathscr{F}_{m,F}^\mathrm{prim}$ for which this permutation action is isomorphic to $G$.  (Note that $\mathrm{Gal}(\widetilde{L}/F)$ must act primitively since the subgroup $\mathrm{Gal}(\widetilde{L}/L)$ is maximal for any $L \in \mathscr{F}_{n,F}^\mathrm{prim}$.)

We then have the following slight refinement of Corollary~\ref{cor:koymans-thorner}.

\begin{corollary}\label{cor:koymans-thorner-group}
    Let $G$ be a primitive permutation group of degree $m$ and let $F$ be a number field.  Then for any integers $\ell \geq 2$ and $r \geq 1$, any $Q \geq 1$, and any $\varepsilon > 0$, there holds
        \[
            \sum_{L \in \mathscr{F}_{m,F}^G(Q)} |\mathrm{Cl}(L)[\ell]|^r
                \ll_{F,m,\ell,r,\varepsilon} Q^{\frac{r}{2}+\varepsilon} \cdot \left( 1 + |\mathscr{F}_{m,F}^G(Q)|^{1 - \frac{r}{\ell(m-1)+1}}\right).
        \]
\end{corollary}
\begin{proof}
    This follows immediately from Theorem~\ref{thm:sparse_bad}, Proposition~\ref{prop:avg_pit} and Proposition~\ref{prop:koymans-thorner} as described above.
\end{proof}

\subsection{Generation of the class group: Proof of Theorem~\ref{thm:approximate-bach}}

   We begin with a general lemma that will be used to show that characters of the class group with order $\ell$ are typically irreducible and faithful when regarded as characters of the Galois group.

    \begin{lemma} \label{lem:induction-irreducible}
    Let $G$ be a finite group, $H$ a maximal subgroup of $G$, and $N$ the maximal normal subgroup of $G$ contained in $H$.  Let $\chi$ be an irreducible primitive character of $H$ (i.e., an irreducible character not induced from any proper subgroup of $H$). Suppose $\chi\big|_N$ is not the restriction of some character of $G$ to $N$ and that $|N|$ and $[G: N]$ are coprime. Then $\textup{Ind}^G_H \chi$ is an irreducible character.
    \end{lemma}

    \begin{proof}
    Suppose $\Ind_H^G \chi$ was not irreducible.
    By Mackey's criterion \cite[Theorem 17.4c]{Huppert98}, we have
    \[\langle \chi^{\tau}, \chi\rangle_{H \cap \tau^{-1} H\tau} \ne 0\]
    for some $\tau$ in $G\backslash H$, where $\chi^{\tau}$ denotes the conjugate representation to $\chi$.
    
    Since $\chi$ is primitive, its restriction to $N$ is a multiple of some irreducible character of $N$ \cite[Corollary 6.12]{Isaacs76}. So this expression can be nonzero only if  the restriction of $\chi$ and $\chi^{\tau}$ to $N$ are equal.
    In this case $\chi$ is preserved under conjugation by $\langle \tau, H\rangle$, which is $G$ since $H$ is maximal. But then \cite[Corollary 8.16]{Isaacs76} and the assumption that $|N|$ and $[G:N]$ are coprime imply that $\chi\big|_{N}$ equals the restriction of some character of $G$.
    \end{proof}

    We are now ready to prove Theorem~\ref{thm:approximate-bach}.

    \begin{proof}[Proof of Theorem~\ref{thm:approximate-bach}]
    Choose $m$, $\ell$, and $Q$ as in the theorem statement.  Fix a positive $\varepsilon < \frac{1}{4n^2 \ell^m m!}$.  Let $H \geq 1$, and suppose that the prime ideals of some $L \in \mathscr{F}_{m,F}^{\mathrm{prim}}(Q)$ with norm at most $H$ generate a proper subgroup of $\mathrm{Cl}(L)/\ell \mathrm{Cl}(L)$.  If so, then there is some class group character $\chi\colon \mathrm{Cl}(L) \to \mathbb{C}^\times$ of order $\ell$ that is trivial on this subgroup.  For this character $\chi$, we would then find that $\chi(\mathfrak{P}) = 1$ for every prime $\mathfrak{P}$ of $L$ with norm at most $H$, and hence
            \begin{equation} \label{eqn:char-sum-if-trivial}
                \sum_{\mathrm{Nm}_{L/\mathbb{Q}} \mathfrak{P} \leq H} \chi(\mathfrak{P})
                    = \pi_L(H).
            \end{equation}
        
        Our goal is therefore to show that \eqref{eqn:char-sum-if-trivial} does not hold for $H = (\log Q)^{A}$ and any character $\chi$ of order $\ell$ and all but $O_{F,n,\ell,\varepsilon}(Q^\varepsilon)$ fields $L \in \mathscr{F}_{n,F}^\mathrm{prim}(Q)$, where $A$ is taken to be  $3 (m!)^2 \ell^{2m} / \varepsilon$.

        As in the proof of Corollary~\ref{cor:ellenberg-venkatesh}, we may assume that $Q \geq Q_0$ for some $Q_0$ depending only on $F$, $n$, $\ell$, and $\varepsilon$.  In fact, let $Q_0$ be the least real number such that:
     $Q_0 \geq \exp( \varepsilon^{-2} \exp(160000 n^2 \ell^{4m} m!^4))$; we may take $C(F,m!) = 400 n (m!)^2$ and $C(F,\ell^m m!) = 400 n \ell^{2m} (m!)^2$ for every $\Delta \geq Q_0$ in Theorem \ref{thm:sparse_bad}; and, for every $H \geq \left(\log Q_0\right)^{\frac{3(m!)^2 \ell^{2m}}{\varepsilon}}$, we have both $\pi_F(H) \geq \frac{1}{2} \frac{H}{\log H}$ and 
        \[
            \left| \sum_{\mathrm{N}\mfp \leq H} \chi(\mfp) \right|
                \leq \frac{H}{4 \log H}
        \]
        for every class group character of $F$ with order $\ell$.                    
        Such a $Q_0$ exists, and depends only on $F$, $m$, $\ell$, and $\varepsilon$.  There are therefore at most $O_{F,m,\ell,\varepsilon}(1)$ fields in $\mathscr{F}_{m,F}^\mathrm{prim}(Q_0)$, all of which we may include in the exceptional set, so we may assume henceforth that all $L \in \mathscr{F}_{m,F}^\mathrm{prim}(Q)$ under consideration have $\Delta_L \geq Q_0$.    

        We begin by setting $\epsilon_1 = \varepsilon / 2 m!$ and letting $\mathcal{E}_1$ be the subset of those $L \in \mathscr{F}_{m,F}^\mathrm{prim}(Q)$ that are not linearly disjoint from the set $\mathbb{X}_\mathrm{bad}(F,\epsilon_1)$.  We begin by claiming that $|\mathcal{E}_1| \ll_m Q^{\varepsilon}$.  Indeed, there are $O_m(1)$ extensions $L \in \mathscr{F}_{m,F}^\mathrm{prim}$ not linearly disjoint from a fixed $K \in \mathbb{X}_\mathrm{bad}(F,\epsilon_1)$, and Theorem~\ref{thm:sparse_bad} with $\Delta = Q^{m!/2}$ and $d=m!$ shows that the number of bad $K$ is at most
            \[
                Q^{\frac{\varepsilon}{4} + \frac{100 n (m!)^2 \cdot \varepsilon}{\sqrt{\log\log Q_0}}} (\log Q)^{800 n (m!)^2}
                    \leq Q^{\frac{\varepsilon}{2} + 800 n (m!)^2 \frac{\log\log Q}{\log Q}}
                    \leq Q^{\varepsilon},
            \]
    where we have used in the first inequality that $\sqrt{\log\log Q_0} \geq 400 n (m!)^2$ and in the second that (say) $\frac{\log\log Q}{\log Q} \leq \frac{1}{\sqrt{\log Q_0}} \leq \varepsilon \cdot\left(1600 n (m!)^2\right)^{-1}$, both of which readily follow from our assumptions on $Q_0$.  Note that 
        \[
            \left(\frac{m!}{2} \log Q \right)^{2 + \frac{m!}{2\epsilon_1}} \leq (\log Q)^{\frac{3 (m!)^2}{\varepsilon}}.
        \]    
    Since we have assumed $\varepsilon < \frac{1}{4n^2 \ell^m m!}$, one computes that $c(\epsilon_1) = \epsilon_1^{1/2} / 18$, and we find for any $H \geq (\log Q)^{3 (m!)^2/\varepsilon}$ that
        \begin{align} \label{eqn:pit-error}
            (m-1) \exp(-c(\epsilon_1) \sqrt{\log H})
                & \leq (m-1) \exp\left( - \frac{\sqrt{3}}{18\sqrt{2}} \frac{1}{\sqrt{m!}} \sqrt{\log\log Q_0}\right) \\
                & \leq (m-1) \exp\left( - \frac{400 \sqrt{3}}{18 \sqrt{2}} (m!)^{3/2} \right)  \notag \\
                & < 10^{-33}. \notag
        \end{align}
    We thus find from Proposition~\ref{prop:avg_pit} and our assumptions on $Q_0$ that for $L \not\in \mathcal{E}_1$, we have
        \begin{equation} \label{eqn:bach-pit}
            \pi_L(H)
                > \left(\frac{1}{2} - 10^{-33}\right) \frac{H}{\log H}
        \end{equation}
    for every $H \geq (\log Q)^{\frac{3 (m!)^2 \ell^{2m}}{\varepsilon}}$.  As stated above, we now wish to contradict this lower bound for every class group character $\chi$ of $L$ with order $\ell$.  

    Thus, suppose that $\chi$ is a nontrivial class group character of order $\ell$ associated with some extension $L \in \mathscr{F}_{m,F}^\mathrm{prim}(Q) \setminus \mathcal{E}_1$.  Let $M/L$ be the associated cyclic degree $\ell$ extension, and let $\widetilde{M}$ denote the normal closure of $M$ over $F$.  Note that $[\widetilde{M}:F] \leq \ell^m m!$ and that $\Delta_{\widetilde{M}} \leq \Delta_F^{\ell^m m!/2}$.  Thus, let $\epsilon_2 = \frac{\varepsilon}{2 \ell^m m!}$ and let $\mathcal{E}_2$ be the subset of $L \in \mathscr{F}_{m,F}^\mathrm{prim}(Q) \setminus \mathcal{E}_1$ for which any of these associated extensions $\widetilde{M}$ lie in $\mathbb{X}_\mathrm{bad}(F,\epsilon_2)$.  As in our treatment of $\mathcal{E}_1$, we observe that $|\mathcal{E}_2| \ll_{m,\ell} Q^{\varepsilon}$.
    
    Thus, suppose that $L \in \mathscr{F}_{m,F}^\mathrm{prim}(Q) \setminus (\mathcal{E}_1 \cup \mathcal{E}_2)$.  Let $\chi$ and $M$ be as above.  Let $G_M = \mathrm{Gal}(\widetilde{M}/F)$ and $H_M = \mathrm{Gal}(\widetilde{M}/L)$.  Then by class field theory, we may regard $\chi$ as a nontrivial linear character of $H_M$, so in particular $\chi$ is a primitive character of $H_M$.  Moreover, the maximal normal subgroup of $G_M$ contained in $H_M$ (i.e., the core of $H_M$) is $\mathrm{Gal}(\widetilde{M}/K) =: N_M$, where $K$ is the normal closure of $L/F$.  In particular, $N_M \simeq C_\ell^r$ for some $r \leq m$, and thus $|N_M|$ and $[G_M:H_M]$ are relatively prime.  

    Finally, to apply Lemma~\ref{lem:induction-irreducible} (as is our goal), we must consider two possibilities.  In particular, either $M$ is not the extension to $L$ of a cyclic degree $\ell$ extension $M_0/F$ (in which $\chi\mid_{N_M}$ is not the restriction to $N_M$ from a character of $G$), or $M$ is such an extension (in which case $\chi$ is such a restriction).  If $M$ is not the extension to $L$ of a cyclic extension $M_0/F$, Lemma~\ref{lem:induction-irreducible} implies that the induction $\chi^* := \mathrm{Ind}_{H_M}^{G_M} \chi$ is irreducible, and it is a faithful character by construction.  Thus, exactly as in \eqref{eqn:pit-error} and the surrounding discussion, we conclude for any $H \geq (\log Q)^{3 \ell^{2m} (m!)^2 / \varepsilon}$ that
        \[
            \left| \sum_{\mathrm{N}\mathfrak{P} \leq H} \chi(\mathfrak{P}) \right|
                < 10^{-33} \frac{H}{\log H}.
        \]
    This contradicts \eqref{eqn:bach-pit}, so it remains to consider those $M$ that are the extension to $L$ of a cyclic extension $M_0/F$.

    In this situation, it will never be the case that the induction $\chi^*$ of $\chi$ to $G_M \simeq C_\ell \times G$ will be irreducible, since $\chi^*$ will simply be the twist of the permutation character $\pi$ of $G$ by a nontrivial cyclic character of $C_\ell$, and the permutation character is not irreducible.  However, by \cite[Lemma 3.9]{LOTZ}, each nontrivial irreducible constituent of $\pi$ is a faithful character of $G$, whence their twists are faithful irreducible characters of $G_M$.  For $L \in \mathscr{F}_{m,F}^\mathrm{prim}(Q) \setminus (\mathcal{E}_1 \cup \mathcal{E}_2)$, these constituents may therefore be treated as before.  The twist of the trivial character, meanwhile, may be regarded as a nontrivial character associated with the cyclic extension $M_0/F$.  Because we have assumed that $\ell \nmid |G|$ and that the extension $M/L$ is unramified, the extension $M_0/F$ must be unramified as well.  In particular, any nontrivial character of the extension $M_0/F$ is a class group character of $F$ of order $\ell$, so by our assumptions on $Q_0$, and analysis analogous to \eqref{eqn:pit-error}, we find in this case that whenever $H \geq (\log Q)^{{3 \ell^{2m} (m!)^2}/{\varepsilon}}$ that
        \[
             \left| \sum_{\mathrm{N}\mathfrak{P} \leq H} \chi(\mathfrak{P}) \right|
                < \left(\frac{1}{4}+10^{-33}\right) \frac{H}{\log H}.
        \]
    This is again sufficient to contradict \eqref{eqn:bach-pit}, completing the proof of the theorem.
    \end{proof}

\subsection{Bounds on Artin $L$-functions: Proof of Corollary~\ref{cor:artin-bound}}

    In this section, we prove Corollary~\ref{cor:artin-bound} concerning bounds on $L(1,\chi)$.

    \begin{lemma}\label{lem:l-one-approx}
        Let $\epsilon > 0$ and let $K/F$ be a nontrivial Galois extension of number fields such that $K$ is not in $\mathbb{X}_\mathrm{bad}(F,\epsilon)$.  Let $\chi$ be a faithful, irreducible character of $\mathrm{Gal}(K/F) =: G$, and let $L(s,\chi)$ denote the associated Artin $L$-function.  Then
            \begin{align*}
                \log L(1,\chi)
                    &= \sum_{\mathrm{N} \mfp \leq (\log |\mathrm{Disc}(K)|)^{2 + \frac{2|G|}{\epsilon}}} \log L_\mfp(1,\chi)  \\
                    & \quad\quad + O_{F,G,\epsilon}\left(\exp\left(-c(\epsilon) \left(2 + \frac{2|G|}{\epsilon}\right)^{1/2} \sqrt{\log\log \Delta_K}\right)\right),
            \end{align*}
        where $c(\epsilon)$ is as in \eqref{eq:cepsilon}
    \end{lemma}
    \begin{proof}
        For convenience, set $c_\epsilon = 2 + \frac{2 |G|}{\epsilon}$.  For any $\sigma > 1$, we may write
            \[
                \log L_\mfp(\sigma, \chi)
                    = - \log\left( 1 - \rho(\sigma_\mfp)|V^{I_\mfp} (\mathrm{N}\mfp)^{-\sigma}\right)
                    = \frac{\chi(\mathrm{Frob}_\mfp)}{\mathrm{N}\mfp^\sigma} + O_G((\mathrm{N}\mfp)^{-2\sigma}).
            \]
        For primes $\mfp$ such that $\mathrm{N}\mfp > (\log |\mathrm{Disc}(K)|)^{c_\epsilon}$, we find by the definition of $\epsilon$-bad fields and partial summation that
            \begin{align*}
                \left|\sum_{\mathrm{N} \mfp > (\log \Delta_K)^{c_\epsilon}} \frac{\chi(\mathrm{Frob}_\mfp)}{\mathrm{N}\mfp^\sigma}\right|
                    & \leq \frac{ \exp\left( -c(\epsilon) c_\epsilon^{1/2} \sqrt{\log\log \Delta_K} \right)}{1-\exp\left( -c(\epsilon) c_\epsilon^{1/2} \sqrt{\log\log \Delta_K} \right)} \\
                    & \ll_{F,G,\epsilon} \exp\left( -c(\epsilon) c_\epsilon^{1/2} \sqrt{\log\log \Delta_K} \right),
            \end{align*}
        where the last inequality follows on observing that the quantity $c(\epsilon)c_\epsilon^{1/2} \sqrt{\log\log \Delta_K}$ may be bounded away from $0$ solely in terms of $F$, $G$, and $\epsilon$.  We also find that
            \[
                \sum_{\mathrm{N} \mfp > (\log \Delta_K)^{c_\epsilon}} \frac{1}{\mathrm{N}\mfp^{2\sigma}}
                    \ll_{F,G} \frac{1}{(\log \Delta_K)^{c_\epsilon}},
            \]
        which is smaller than the claimed bound.
        Upon taking the limit as $\sigma \to 1$, the result follows.
    \end{proof}

    Using this, we are able to provide a proof of Corollary~\ref{cor:artin-bound}.

    \begin{proof}[Proof of Corllary~\ref{cor:artin-bound}]
        Let $\epsilon = \frac{\varepsilon}{2}$, and assume that $K$ does not lie in $\mathbb{X}_\mathrm{bad}(F, \epsilon)$ and that $\chi$ is a faithful, irreducible character of $\mathrm{Gal}(K/F)=:G$.  By a slight abuse of notation, for any prime $\mfp$, we write $\chi(\sigma_\mfp)$ for the sum of the local roots of $L_\mfp(s,\chi)$.  (We caution that if $\mfp$ is ramified in $K/F$, $\chi(\sigma_{\mfp})$ does not have to be a literal character value of $\chi$.)  As in the proof of Lemma~\ref{lem:l-one-approx}, let $c_\epsilon = 2 + \frac{2|G|}{\epsilon}$.  By Lemma~\ref{lem:l-one-approx}, we then find
            \begin{align*}
                \log |L(1,\chi)|
                    &= \sum_{\mathrm{N}\mfp \leq (\log Q)^{c_\epsilon}} \log |L_\mfp(1,\chi)| + O_{F,G,\varepsilon}(\exp\left(-c(\epsilon) c_\epsilon^{1/2} \sqrt{\log \log \Delta_K}\right))\\
                    &= \sum_{\mathrm{N}\mfp \leq (\log Q)^{c_\epsilon}} \frac{\Re(\chi(\sigma_p))}{\mathrm{N}\mfp} + O_{F,G,\varepsilon}(1) \\
                    &\leq \chi(1) \log\log\log Q + O_{F,G,\varepsilon}(1),
            \end{align*}
        where in the last line we have used the prime ideal theorem (or really Mertens' theorem) for $F$.  The upper bound $L(1,\chi) \ll_{F,G,\varepsilon} (\log\log Q)^{\chi(1)}$ follows.  For the lower bound, we find it convenient to define $a(\chi) = \min\{ \Re(\chi(g)) : g \in G\}$.  Proceeding analogously to the above, we obtain
            \[
                \log |L(1,\chi)|
                    \geq a(\chi) \log\log\log Q + \sum_{\mfp \mid \mathfrak{D}_{K/F}} \frac{\Re(\chi(\sigma_\mfp) - a(\chi))}{\mathrm{N} \mfp} + O_{F,G,\varepsilon}(1).
            \]
        For ramified primes, we note that $\chi(\sigma_\mfp)$ is by definition the trace of $\sigma_\mfp$ acting on $V^{I_\mfp}$.  This is equal to the average of $\chi(g)$ over $g \in \sigma_\mfp I_\mfp$, so $\Re(\chi(\sigma_\mfp) - a(\chi)) \geq 0$.  Thus, we may omit the sum over ramified primes above, and the lower bound follows.  Thus, we have proven the claim for all $K \not\in \mathbb{X}_\mathrm{bad}(F,\epsilon)$, and as Theorem~\ref{thm:sparse_bad} shows that
            \[
                \#\{ K \in \mathbb{X}_\mathrm{bad}(F,\epsilon) : \Delta_K \leq Q\}
                    \leq Q^{\epsilon (1+\delta)} (\log Q)^{C(F,|G|)}
                    \ll_{F,G,\varepsilon} Q^{\varepsilon},
            \]
        the result follows.
    \end{proof}

\bibliographystyle{alpha}
\bibliography{references}

\end{document}